\documentclass[a4paper,12pt,reqno]{amsart}
\usepackage[body={17cm,21cm}]{geometry}
\usepackage[initials]{amsrefs}

\usepackage{amssymb,amscd}
\usepackage[all]{xy}

\usepackage[mathscr]{eucal}
\usepackage{enumerate}

\numberwithin{equation}{section}
\theoremstyle{plain}
\newtheorem{thm}[equation]{Theorem}
\newtheorem{cor}[equation]{Corollary}
\newtheorem{lem}[equation]{Lemma}
\newtheorem{prop}[equation]{Proposition}

\theoremstyle{definition}

\theoremstyle{remark}

\newtheorem{claim}[equation]{Claim}

\DeclareMathOperator{\id}{id}

\DeclareFontFamily{U}{wncy}{}
\DeclareFontShape{U}{wncy}{m}{n}{%
<5>wncyr5%
<6>wncyr6%
<7>wncyr7%
<8>wncyr8%
<9>wncyr9%
<10>wncyr10%
<11>wncyr10%
<12>wncyr6%
<14>wncyr7%
<17>wncyr8%
<20>wncyr10%
<25>wncyr10}{}
\DeclareMathAlphabet{\cyr}{U}{wncy}{m}{n}
\begin{document}

\title[Bounding diameter of conical K\"{a}hler metric]
{Bounding diameter of conical K\"{a}hler metric}

\author{Yan Li}

\address{Yan Li \newline School of Mathematical Sciences, \newline Capital Normal University,
\newline 105 Xisanhuanbeilu, \newline 100048 Beijing, China}
\email{liyandota@hotmail.com}

\date{\today}

\maketitle

\begin{abstract}
In this paper we research the differential geometric and algebro-geometric properties of the noncollasping limit in the
conical continuity equation which generalize the theory in \cite{NTZ}.

\end{abstract}

\section{Introduction}
The Ricci flow proposed by Hamilton in \cite{H} has been one of the most powerful tools in geometric analysis with the solution of Poincar\'{e} conjecture. Similarly, the K\"{a}hler Ricci flow has also become an fundamental tool in the study of K\"{a}hler geometry for many years. J. Song, G. Tian and their collaborators (\cite{TZ}, \cite{ST1}, \cite{SW1}, \cite{SW2}, \cite{SY}, \cite{ST2}, \cite{ST3}, \cite{T}) developed the Analytic Minimal Model Program through K\"{a}hler Ricci flow. However, in studying the singularity formation of the K：ahler-Ricci flow there are some difficulties because we do not know how to control the lower bound for the Ricci curvature along the flow. To overcome these difficulties, in \cite{NT}, G. La Nave and G. Tian introduced a new continuity equation. In \cite{NTZ}, G. La Nave, G. Tian and Z. L. Zhang investigated the differential geometric and algebro-geometric properties of the noncollapsing limit in the continuity method. Properties of the continuity equation in \cite{NT} are very similar with properties of the K\"{a}hler Ricci flow.

On the other hand, conical K\"{a}hler-Einstein metric plays an essential role in recent great progress about Yau-Tian-Donaldson conjecture, see \cite{CDS1} \cite{CDS2} \cite{CDS3} \cite{T1}. In \cite{CGP} and \cite{GP}, H. Guenancia and M. P\'{a}un constructed smooth approximation of conical metric. Recently, L. M. Shen in \cite{S} got a result that is about maximal time existence of unnormalized conical K\"{a}hler Ricci flow. Therefore, a natural problem is that what properties the conical version of continuity equations have.

In this paper we generalize the theory in \cite{NTZ} to the conical version and research the differential and algebro-geometric properties of the limit in the conical continuity equation. We will focus on the noncollasping case.

To begin with, we assume that $M$ is a projective manifold with a K\"{a}hler metric $\omega_{0}\in c_{1}(L^{'})$, where $L^{'}$ is a line bundle on $M$. Let $D$ be a smooth hypersurface in $M$ and $\beta \in (0,1)$. We consider the 1-parameter family of equations:
\begin{equation}
    \omega=\omega_{0}-t(Ric(\omega)-(1-\beta)[D]),
\end{equation}
where $[D]$ is the current of integration along $D$. Clearly, the K\"{a}hler classes vary according to the linear relation: $[\omega]=[\omega_{0}]-t(c_{1}(M)-(1-\beta)c_{1}(L_{D}))$, where $[\omega]$ denotes the K\"{a}hler class of $\omega$ and
$c_{1}(L_{D})$ denotes the first Chern class of line bundle $L_{D}$ associated with hypersurface $D$.

Our first theorem is:

\begin{thm}
For any initial K\"{a}hler metric $\omega_{0}$, there is a unique singular family of solution $\omega_{t}$ for (1.1) on $M\times[0,T)$,where
\begin{equation}
   T=sup\{t|[\omega_{0}]-t(c_{1}(M)-(1-\beta)c_{1}(L_{D}))>0\}.
\end{equation}
and each $\omega_{t}$ is a conical metric.
\end{thm}

If $T<\infty$, we need to examine the limit of $\omega_{t}$ as $t$ tends to $T$. We have the following result if  $\omega_{t}$ is noncollapsing.
\begin{thm}
Assume that $([\omega_{0}]-T(c_{1}(M)-(1-\beta)c_{1}(L_{D})))^{n}>0$, where $n=dim_{\mathbb{C}}M$, then $\omega_{t}$ converge to a unique weakly K\"{a}hler metric $\omega_{T}$ such that $\omega_{T}$ is smooth on $M\backslash (\mathcal{S}_{M}\cup D)$ and satisfies:
$$
  \omega_{T}=\omega_{0}-TRic(\omega_{T}), \ on \ M\backslash (\mathcal{S}_{M}\cup D),
$$
where
$$
\mathcal{S}_{M}=\bigcap\{F|F \ is \ a \ disivior \ satisfying \ [\omega_{0}]-T(c_{1}(M)-(1-\beta)c_{1}(L_{D}))-\rho[F]>0 \ for \ some \ \rho>0  \}.
$$
\end{thm}

In \cite{NTZ}, the limit space which $(M,\omega_{t})$ converge to in the Gromov-Hausdorff topology has more regular properties, such as metric structure, algebraic structure. In the conical situation, we also have similar properties.
\begin{thm}
Assume as in above theorem, $\beta\in \mathbb{Q}\cap(0,1)$ and $c_{1}(L_{D})$ is semi-positive, then
\begin{enumerate}
\item  $(M,\omega_{t})$ converges in the Gromov-Hausdorff topology to a compact path metric space $(M_{T},d_{T})$ which is the metric completion $(M\backslash (\mathcal{S}_{M}\cup D),\omega_{T})$;  \\
\item  $M_{T}$ has regular part and singular part, i.e. $M_{T}=\mathcal{R}\cup \mathcal{S}$, a point $x \in   \mathcal{R}$ if and only if the tangent cone at $x$ is $\mathbb{C}^{n}$;\\
\item  $\mathcal{S}$ is closed and has real codimension $\geq 2$ and $\mathcal{R}$ is geodesically convex; \\
\item  $M_{T}$ is homeomorphic to a normal projective variety with $\mathcal{S}$ corresponding to a subvariety.
\end{enumerate}
\end{thm}
\textbf{Acknowledgement:}The author would like to thank his advisor Prof. F.Q.Fang, for his
constant help, support and encouragement over the years. He also wants to thank Prof. Z.l.Zhang for his careful reading of a preprint of
this paper and many helpful suggestions and discussions. Finally, he thanks Prof. Y.Yuan for his enthusiastic discussion.

\section{Existence and uniqueness of conical continuity equation}
\subsection{Proof of Theorem 1.2}
 First we reduce (1.1) to a scalar equation. Choose a real closed (1, 1) form $\psi$  representing $c_{1}(M)$ and a smooth volume form $\Omega$ such that $Ric(\Omega)=\psi$. Let $L_{D}$ be a line bundle with a Hermitian metric $h_{D}$ and $s_{D}$ be a defining section of $L_{D}$. $\Theta_{h_{D}}$ represents the curvature of $L_{D}$. By Poincar\'{e}-Lelong formula, we have

$$
  \sqrt{-1}\partial\overline{\partial}\log|s_{D}|^{2}_{h_{D}}=-\Theta_{h_{D}}+[D].
$$

Set $\widetilde{\omega_{t}}=\omega_{0}-t(\psi-(1-\beta)\Theta_{h_{D}})$ for $t \in [0,T)$. One can easily show that $\omega_{t}=\widetilde{\omega_{t}}+t\sqrt{-1}\partial\overline{\partial}u$ satisfies (1.1) if $u$ satisfies
\begin{equation}
(\widetilde{\omega_{t}}+t\sqrt{-1}\partial\overline{\partial}u)^{n}=e^{u}\frac{\Omega}{|s_{D}|^{2(1-\beta)}_{h_{D}}},
\end{equation}
where $\widetilde{\omega_{t}}+t\sqrt{-1}\partial\overline{\partial}u>0$.
\begin{prop}
The equation (2.1) is solvable for each $t \in [0, T)$.
\end{prop}
This proposition is obvious according to Theorem A in \cite{GP}.\\

In \cite{GP}, H. Guenancia and M. P\'{a}un introduced that for any $\epsilon>0$, the function $\chi_{\beta}:[\epsilon^2,\infty)\rightarrow \mathbb{R}$ defined
as follows:
$$
 \chi_{\beta}(\epsilon^2+t)=\beta\int_{0}^{r}\frac{(\epsilon^2+r)^{\beta}}{r}
$$
for any $t\geq 0$. There exists a constant $C$ such that $0\leq\chi_{\beta}(t)\leq C$ provided that $t$ belongs to a bounded interval. This function is useful to prove Theorem (1.4).

To prove the uniqueness we argue as Jeffery \cite{Je}.
\begin{prop}
Assume that $u$ is a solution of equation(2.1) such that $\omega_{t}=\widetilde{\omega_{t}}+t\sqrt{-1}\partial\overline{\partial}u$ is a conical metric. Then $u$ is unique.
\end{prop}
\begin{proof}
Assume $u_1$ and $u_2$ are solutions of equation(2.1). Set $v=u_1-u_2$. One immediately obtains  the following equation.
$$
   e^{v}\omega_1^n=(\omega_1+t\sqrt{-1}\partial\overline{\partial}v)^n,
$$
where $\omega_1=\widetilde{\omega_t}+t\sqrt{-1}\partial\overline{\partial}u_1$.
Let $F_k=\frac{1}{k}|s_D|^{2p}_{h_D}(2p<\beta)$ and $v_k=v+F_k$. It is easy to show
$$
  \sqrt{-1}\partial\overline{\partial}|s_D|^{2p}_{h_D}\geq -p|s_D|^{2p}_{h_D}\Theta_{h_D}.
$$
For each $k$, $v_k$ can attain maximum at $P_k \in M\backslash D$. Then at $P_k$ one knows
$$
 e^v(\det g^1_{i\bar{j}})=\det (g^1_{i\bar{j}}+\sqrt{-1}\partial_i\overline{\partial}_{\bar{j}}(v_k-F_k)).
$$
Choose normal coordinate at $P_k$ which simultaneously diagonalize $(g^1_{i\bar{j}})$ and $(\sqrt{-1}\partial_i\overline{\partial}_{\bar{j}}(v_k-F_k))$, i.e. $g^1_{i\bar{j}}(P_k)=\delta_{ij}$ and
$\sqrt{-1}\partial_i\overline{\partial}_{\bar{j}}(v_k-F_k)(P_k)=\delta_{ij}((v_k)_{i\bar{j}}-(F_k)_{i\bar{j}})$.
Notice that $(v_k)_{i\bar{i}}(P_k)\leq 0$. Then one has
\begin{align*}
  e^{v(P_k)} &= \prod_{i=1}^{k}(1+(v_k)_{i\bar{i}}-(F_k)_{i\bar{i}})  \\
             &\leq \prod_{i=1}^{k}(1-(F_k)_{i\bar{i}})\leq \prod_{i=1}^k(1+\frac{C}{k}(\Theta_{h_D})_{i\bar{i}})\\
             &\leq \prod_{i=1}^{k}(1+\frac{A}{k})^n.
\end{align*}
Let $k\rightarrow \infty$, one obtains
$$
   v\leq 0
$$
By the similar argument, one has $v\geq 0$. Therefore, $u_1=u_2$.
\end{proof}

\subsection{Proof of Theorem 1.4}
In this subsection we investigate the regular properties of limit metric.
\begin{lem}
Let $F$ be a divisor in a projective manifold $M$. If $F$ is big, then there is an effective divisor $E$ such that $[F]-\epsilon[E]>0$ for all sufficiently small $\epsilon>0$.
\end{lem}
By the assumption of Theorem(1.4) one knows that $[\omega_{0}]-T(c_{1}(M)-(1-\beta)c_{1}(L_{D}))$ is big, then by the above Lemma there is a effective divisor $E$ such that$[\omega_{0}]-T(c_{1}(M)-(1-\beta)c_{1}(L_{D}))-\iota[L_E]$ is ample for some $\iota$. Let $h_E$ be a Hermitian metric on $L_E$ and $\sigma_E$ be a defining section of $E$. Thus by the ampleness of $[\omega_{0}]-T(c_{1}(M)-(1-\beta)c_{1}(L_{D}))-\iota[L_E]$, one knows
$$
  \widetilde{\omega_T}-\iota Ric(h_E)>0.
$$

Now we begin to prove Theorem(1.4).
\begin{proof}
Let $\widetilde{\omega_{t, E}}=\widetilde{\omega_t}+\iota \sqrt{-1}\partial\overline{\partial}\log|\sigma_E|^2_{h_E}$. If $\bar{t}$ is sufficiently small, $\widetilde{\omega_{t, E}}$ is a smooth K\"{a}hler metric on $M\backslash E$ for each $t \in [T-\bar{t}, T+\bar{t}]$. Set $\psi_{\epsilon}=\delta \chi(\epsilon^2+|s_{D}|^{2}_{h_{D}})$ and $\widetilde{\omega_{t, E, \epsilon}}=\widetilde{\omega_{t, E}}+\sqrt{-1}\partial\overline{\partial}\psi_\epsilon$. If $\delta$ is sufficiently small, $\widetilde{\omega_{t, E, \epsilon}}$ is also a smooth K\"{a}hler metric on $M\backslash E$ for all $\epsilon$ and $t \in [T-\bar{t}, T]$. Now we consider the following approximation equation
$$
 (\widetilde{\omega_{t, E, \epsilon}}+\sqrt{-1}\partial\overline{\partial}(tv_{\epsilon}-\iota\log|\sigma_E|^2_{h_E}))^n
 =e^{\frac{\psi_{\epsilon}}{t}+v_{\epsilon}}\frac{\Omega}{(\epsilon^{2}+|s_D|^{2}_{h_D})^{1-\beta}}.
$$
Set $w_{\epsilon}=tv_{\epsilon}-\iota\log|\sigma_E|^2_{h_E}$. Assume that $w_{\epsilon}$ attains minimum at $y_0$, one has
$$
  \frac{\psi_{\epsilon}}{t}+\frac{1}{t}(w_{\epsilon}+\iota\log|\sigma_E|^2_{h_E})\geq \log\frac{(\epsilon^{2}+|s_D|^{2}_{h_D})^{1-\beta}(\widetilde{\omega_{t, E, \epsilon}})^n}{\Omega}\geq -C.
$$
Therefore
$$
  w_{\epsilon}\geq -C-\iota\log|\sigma_E|^2_{h_E}\geq -C.
$$
For the upper bound of $w_{\epsilon}$, one needs to consider the following equation. For $t\in[T-\bar{t}, T)$
$$
(\widetilde{\omega_{t}}+\sqrt{-1}\partial\overline{\partial}\psi_{\epsilon}+t\sqrt{-1}\partial\overline{\partial}v_{\epsilon})^n
=e^{\frac{\psi_{\epsilon}}{t}+v_{\epsilon}}\frac{\Omega}{(\epsilon^{2}+|s_D|^{2}_{h_D})^{1-\beta}}.
$$
Although $\widetilde{\omega_{t}}+\sqrt{-1}\partial\overline{\partial}\psi_{\epsilon}$ may not be a K\"{a}hler metric near $T$, it can be controlled from above,
we still make use of maximum principle to get
$$\sup v_{\epsilon}\leq C.$$
Combining above consequences, one obtains $$-C\leq w_{\epsilon}\leq C-\iota\log|\sigma_E|^2_{h_E}.$$  \\
Set $ \omega_{t, E, \epsilon}=\widetilde{\omega_{t, E, \epsilon}}+\sqrt{-1}\partial\overline{\partial}w_{\epsilon}$ and $t_0=T-\bar{t}$.
\begin{claim}
For $t\in[t_0, T]$, there exist two constants $C$ and $\alpha$ which are independent of $t$ and $\epsilon$ such that
$$
 C^{-1}|\sigma_E|^{2\alpha(n-1)+\frac{2\iota}{t}}_{h_E}\widetilde{\omega_{t_0, E, \epsilon}}\leq \omega_{t, E, \epsilon}
 \leq \frac{C}{|\sigma_E|^{2\alpha}_{h_E}}\widetilde{\omega_{t_0, E, \epsilon}}.
$$
\end{claim}
\begin{proof}
Set $F_{t_0, D, \epsilon}=\log \frac{\Omega}{(\epsilon^{2}+|s_D|^{2}_{h_D})^{1-\beta}(\widetilde{\omega_{t_0, E, \epsilon}})^n}$. By Yau's Schwarz lemma \cite{Y}, one deduces
$$
\triangle_{\omega_{t,E,\epsilon}}\log tr_{\widetilde{\omega_{t_0, E, \epsilon}}}\omega_{t,E,\epsilon}\geq \frac{1}{tr_{\widetilde{\omega_{t_0, E, \epsilon}}}
\omega_{t,E,\epsilon}}(-g^{i\bar{j}}(\widetilde{\omega_{t_0, E, \epsilon}})R_{i\bar{j}}(\omega_{t,E,\epsilon})+g^{i\bar{j}}(\omega_{t,E,\epsilon})
g_{k\bar{l}}(\omega_{t,E,\epsilon})R_{i\bar{j}}^{ \ \ k\bar{l}}(\widetilde{\omega_{t_0, E, \epsilon}})).
$$
Now we take an holomorphic orthonormal coordinates at a point $(t,p)$ such that $g_{i\bar{j}}(\widetilde{\omega_{t_0, E, \epsilon}})=\delta_{ij}$, and $g_{i\bar{j}}(\omega_{t,E,\epsilon})=\lambda_i\delta_{ij}$. So we have
$$
-g^{i\bar{j}}(\widetilde{\omega_{t_0, E, \epsilon}})R_{i\bar{j}}(\omega_{t,E,\epsilon})=\triangle_{\widetilde{\omega_{t_0, E, \epsilon}}}(\frac{\psi_{\epsilon}}{t}+\frac{1}{t}(w_{\epsilon}+\iota\cdot \log|\sigma_E|^2_{h_E})+F_{t_0,D,\epsilon})-\sum_{i,k}R_{i\bar{i}k\bar{k}}(\widetilde{\omega_{t_0, E, \epsilon}}),
$$
and
$$
g^{i\bar{j}}(\omega_{t,E,\epsilon})
g_{k\bar{l}}(\omega_{t,E,\epsilon})R_{i\bar{j}}^{ \ \ k\bar{l}}(\widetilde{\omega_{t_0, E, \epsilon}})=\sum_{i,k}\frac{\lambda_k}{\lambda_i}R_{i\bar{i}k\bar{k}}(\widetilde{\omega_{t_0, E, \epsilon}}).
$$
Thus we have
\begin{multline}
\triangle_{\omega_{t,E,\epsilon}}\log tr_{\widetilde{\omega_{t_0, E, \epsilon}}}\omega_{t,E,\epsilon}\geq   \\
\frac{1}{\sum_{m}\lambda_m}\{\sum_{i<k}
(\frac{\lambda_i}{\lambda_k}+\frac{\lambda_k}{\lambda_i}-2)R_{i\bar{i}k\bar{k}}(\widetilde{\omega_{t_0, E, \epsilon}})+\triangle_{\widetilde{\omega_{t_0, E, \epsilon}}}(\frac{\psi_{\epsilon}}{t}+\frac{1}{t}(w_{\epsilon}+\iota\cdot \log|\sigma_E|^2_{h_E})+F_{t_0,D,\epsilon})\}
\end{multline}
The following result is contained in \cite{GP}. We denote $\Psi_{\epsilon,\rho}=C\chi_{\rho}(\epsilon^2+|s_D|^2_{h_D})$ and there exist constants $C$ and $0<\rho<1$ such that
$$
R_{i\bar{i}k\bar{k}}(\widetilde{\omega_{t_0, E, \epsilon}})\geq -(C+(\Psi_{\epsilon,\rho})_{i\bar{i}}).
$$
Using the symmetry of the curvature tensor, we also have
$$
R_{i\bar{i}k\bar{k}}(\widetilde{\omega_{t_0, E, \epsilon}})\geq -(C+(\Psi_{\epsilon,\rho})_{k\bar{k}}).
$$
Notice that
$$
\frac{1}{\sum_{m}\lambda_m}\sum_{i<k}
(\frac{\lambda_i}{\lambda_k}+\frac{\lambda_k}{\lambda_i}-2)R_{i\bar{i}k\bar{k}}(\widetilde{\omega_{t_0, E, \epsilon}})\geq -\frac{1}{\sum_{m}\lambda_m}
\sum_{i<k}\{\frac{\lambda_k}{\lambda_i}(C+(\Psi_{\epsilon,\rho})_{i\bar{i}})+\frac{\lambda_i}{\lambda_k}(C+(\Psi_{\epsilon,\rho})_{k\bar{k}})\}
$$
and
$$
\triangle_{\omega_{t,E,\epsilon}}(\Psi_{\epsilon,\rho})= \sum_{i}\frac{(\Psi_{\epsilon,\rho})_{i\bar{i}}}{\lambda_i} \geq
\frac{1}{\sum_{m}\lambda_m}
\sum_{i<k}\{\frac{\lambda_k}{\lambda_i}(C+(\Psi_{\epsilon,\rho})_{i\bar{i}})+\frac{\lambda_i}{\lambda_k}(C+(\Psi_{\epsilon,\rho})_{k\bar{k}})\}-C
tr_{\omega_{t,E,\epsilon}}\widetilde{\omega_{t_0,E,\epsilon}}.
$$
Therefore one gets
\begin{multline}
\triangle_{\omega_{t,E,\epsilon}}\log (tr_{\widetilde{\omega_{t_0, E, \epsilon}}}\omega_{t,E,\epsilon}+\Psi_{\epsilon,\rho})\geq  \\
\frac{1}{tr_{\widetilde{\omega_{t_0,E,\epsilon}}}\omega_{t,E,\epsilon}}\triangle_{\widetilde{\omega_{t_0, E, \epsilon}}}(\frac{\psi_{\epsilon}}{t}+\frac{1}{t}(w_{\epsilon}+\iota\cdot \log|\sigma_E|^2_{h_E})+F_{t_0,D,\epsilon})-C
tr_{\omega_{t,E,\epsilon}}\widetilde{\omega_{t_0,E,\epsilon}}.
\end{multline}
From \cite{GP} one knows
$$
\sqrt{-1}\partial\bar{\partial}F_{t_0.D,\epsilon}\geq -(C\widetilde{\omega_{t_0,E,\epsilon}}+\sqrt{-1}\partial\bar{\partial}\Psi_{\epsilon,\rho}); \
|F_{t_0,D,\epsilon}|_{C^0}\leq C.
$$
By taking the trace with respect to $\widetilde{\omega_{t_0,E,\epsilon}}$, we get
$$
\triangle_{\widetilde{\omega_{t_0,E,\epsilon}}}F_{t_0.D,\epsilon}\geq -nC-\triangle_{\widetilde{\omega_{t_0,E,\epsilon}}}\Psi_{\epsilon,\rho}.
$$
Taking a simple calculation, one has
\begin{align*}
\triangle_{\widetilde{\omega_{t_0,E,\epsilon}}}\Psi_{\epsilon,\rho}=\sum_{i}\frac{(\Psi_{\epsilon,\rho})_{i\bar{i}}}{\lambda_i}& \geq
\frac{\triangle_{\widetilde{\omega_{t_0,E,\epsilon}}}(\Psi_{\epsilon,\rho})+nC}{tr_{\widetilde{\omega_{t_0,E,\epsilon}}}\omega_{t,E,\epsilon}}
-Ctr_{\omega_{t,E,\epsilon}}\widetilde{\omega_{t_0,E,\epsilon}}  \\
&\geq -\frac{\triangle_{\widetilde{\omega_{t_0,E,\epsilon}}}F_{t_0.D,\epsilon}}{tr_{\widetilde{\omega_{t_0,E,\epsilon}}}\omega_{t,E,\epsilon}}
-Ctr_{\omega_{t,E,\epsilon}}\widetilde{\omega_{t_0,E,\epsilon}}.
\end{align*}
Note that
$$
\triangle_{\widetilde{\omega_{t_0, E, \epsilon}}}(\frac{\psi_{\epsilon}}{t}+\frac{1}{t}(w_{\epsilon}+\iota\cdot \log|\sigma_E|^2_{h_E})+F_{t_0,D,\epsilon})
=\frac{1}{t}tr_{\widetilde{\omega_{t_0,E,\epsilon}}}(\omega_{t,E,\epsilon}-\widetilde{\omega_t})\geq -\frac{1}{t}tr_{\widetilde{\omega_{t_0,E,\epsilon}}}\widetilde{\omega_t}\geq -C,
$$
the last inequality bases on $\widetilde{\omega_{t_0,E,\epsilon}}\geq C^{-1}\widetilde{\omega_t}$.
There is an easy fact that is
$$
(tr_{\omega_{t,E,\epsilon}}\widetilde{\omega_{t_0,E,\epsilon}})(tr_{\widetilde{\omega_{t_0,E,\epsilon}}}\omega_{t,E,\epsilon})\geq n.
$$
Notice that there exists a constant $C'$ such that $\widetilde{\omega_{t,E,\epsilon}}\geq C'\widetilde{\omega_{t_0,E,\epsilon}}$.

We denote
$H=\log(tr_{\widetilde{\omega_{t_0,E,\epsilon}}}\omega_{t,E,\epsilon}+2\Psi_{\epsilon,\rho})-\frac{(1+C)}{C'}w_{\epsilon}$, then by calculation one has
\begin{align*}
\triangle_{\omega_{t,E,\epsilon}}H & \geq -Ctr_{\omega_{t,E,\epsilon}}\widetilde{\omega_{t_0,E,\epsilon}}-\frac{(1+C)}{C'}tr_{\omega_{t,E,\epsilon}}
(\omega_{t,E,\epsilon}-\widetilde{\omega_{t,E,\epsilon}})  \\
& \geq tr_{\omega_{t,E,\epsilon}}\widetilde{\omega_{t_0,E,\epsilon}}-n(1+C).
\end{align*}
Assume $H$ attains maximum at $x_0$, one deduces
$$
tr_{\omega_{t,E,\epsilon}}\widetilde{\omega_{t_0,E,\epsilon}}(x_0)\leq C.
$$
Notice that
$$
tr_{\widetilde{\omega_{t_0,E,\epsilon}}}\omega_{t,E,\epsilon}(x_0)\leq (tr_{\omega_{t,E,\epsilon}}\widetilde{\omega_{t_0,E,\epsilon}}(x_0))^{n-1}\cdot
e^{\frac{\psi_{\epsilon}}{t}+\frac{1}{t}(w_{\epsilon}+\iota\cdot \log|\sigma_E|^2_{h_E})+F_{t_0,D,\epsilon}}(x_0)\leq C.
$$
Therefore according to the estimate of $w_{\epsilon}$ and the boundness of $\Psi_{\epsilon,\rho}$, there exist constants $C$ and $\alpha$ such that
\begin{align*}
\log tr_{\widetilde{\omega_{t_0,E,\epsilon}}}\omega_{t,E,\epsilon}&\leq \log tr_{\widetilde{\omega_{t_0,E,\epsilon}}}\omega_{t,E,\epsilon}(x_0)-\frac{(1+C)}{C'}
w_{\epsilon}(x_0)+\frac{(1+C)}{C'} w_{\epsilon}+C  \\
& \leq C-\alpha\log|\sigma_E|^2_{h_E}.
\end{align*}
Furthermore one gets
$$
tr_{\widetilde{\omega_{t_0,E,\epsilon}}}\omega_{t,E,\epsilon}\leq \frac{C}{|\sigma_E|^{2\alpha}_{h_E}}.
$$
By the similar argument one has
$$
tr_{\omega_{t,E,\epsilon}}\widetilde{\omega_{t_0,E,\epsilon}}\leq (tr_{\widetilde{\omega_{t_0,E,\epsilon}}}\omega_{t,E,\epsilon})^{n-1}\cdot
e^{-(\frac{\psi_{\epsilon}}{t}+\frac{1}{t}(w_{\epsilon}+\iota\cdot \log|\sigma_E|^2_{h_E})+F_{t_0,D,\epsilon})}\leq
\frac{C}{|\sigma_E|_{h_E}^{2\alpha(n-1)+\frac{2\iota}{t}}}.
$$
\end{proof}
By the above Claim, one knows that for any compact subset $K \subset M\backslash (D\cup E)$, there exists a constant $C_{K}>0$ independent of $\epsilon$ and $t$ such that $C_{K}^{-1}\omega_{0}\leq \omega_{t,E,\epsilon}\leq C_{K} \omega_{0}$, i.e. $|\triangle_{\omega_0}w_{\epsilon}|\leq C$. By theorem 17.14 in \cite{GT}, we have that $|w_{\epsilon}|_{C^{2, \alpha}}\leq C^{'}_{K}$ on $K\times [T-\bar{t},T]$. Furthermore, by the standard bootstrapping argument one has that for any $l>0$, $|w_{\epsilon}|_{C^{l, \alpha}}\leq C_{K,l}$ on $K\times [T-\bar{t},T]$. By the
standard diagonal argument and passing to a subsequence, we see that $w_{\epsilon_i,t_i}$ $C^{\infty}$ converges to a $(1,1)$ form on each compact subset $K \subset M\backslash (D\cup E)$ when $\epsilon_i\rightarrow 0$ and $t_i\rightarrow T$. Back to equation(2.1), we know that there exists a subsequence $\{t_i\}_{i=1}^{\infty}$ such that $u_{t_i}$ $C^{\infty}$
converges to $u_T$ on each compact subset $K \subset M\backslash (D\cup E)$ when $t_i\rightarrow T$. A priori, this limit may not be unique. So we still need to prove that $u_T$ is unique, i.e., independent of the subsequence $\{t_i\}_{i=1}^{\infty}$. \\

Differentiating $t$ at both sides of equation (2.1), one has
$$
 \triangle_{\omega_t}\dot{u_t}=\frac{1}{t}\dot{u_t}-\frac{n}{t^2}+\frac{1}{t^2}tr_{\omega_t}\omega_0.
$$
By the simple calculation, one gets
$$
  \triangle_{\omega_t}(u_t-n\log t)^{'}\geq \frac{1}{t}(u_t-n\log t)^{'}.
$$
By maximum principle one knows
$$
\frac{1}{t}(u_t-n\log t)^{'}\leq 0.
$$
i.e. $(u_t-n\log t)$ is decreasing as $t\rightarrow T$. Combining the previous argument, we see that $u_T$ is unique.
Therefore Theorem(1.4) is proved.
\end{proof}
\subsection{Smooth approximation of metric with conical singularities}
In this subsection, we assume $L_D$ is a semi-positive line bundle, i.e. there exists a Hermitian metric $h_D$ such that the curvature $\Theta_{h_D}\geq 0$. Fix $t\in[T-\bar{t}, T)$, we consider the approximation equation
\begin{equation}
(\omega_{t,\epsilon})^n=e^{\psi_{\epsilon}+v_{\epsilon}}\frac{\Omega}{(\epsilon^2+|s_D|^2_{h_D})^{1-\beta}},
\end{equation}
 where
$\omega_{t,\epsilon}=\widetilde{\omega_t}+t\sqrt{-1}\partial\overline{\partial}\psi_{\epsilon}
+t\sqrt{-1}\partial\overline{\partial}v_{\epsilon}$. By the calculation, one has
\begin{align*}
Ric(\omega_{t,\epsilon})&=-\frac{1}{t}(\omega_{t,\epsilon}-\widetilde{\omega_t})+Ric(\Omega)+(1-\beta)
\sqrt{-1}\partial\overline{\partial}\log(\epsilon^2+|s_D|^2_{h_D})  \\
&= -\frac{1}{t}\omega_{t,\epsilon}+\frac{1}{t}\omega_0+(1-\beta)\frac{\epsilon^2\langle\nabla s,\overline{\nabla s}\rangle}{(\epsilon^2+|s_D|^2_{h_D})^2}+(1-\beta)\frac{\epsilon^2}{\epsilon^2+|s_D|^2_{h_D}}\Theta_{h_D} \\
& \geq -\frac{1}{t}\omega_{t,\epsilon}.
\end{align*}
For fixed $t\in[T-\bar{t}, T)$, by Claim(2.3) we know
$$
  C_t^{-1}\omega_0\leq A_t^{-1}\widetilde{\omega_{t,\epsilon}}\leq \omega_{t,\epsilon}\leq A_t\widetilde{\omega_{t,\epsilon}}
  \leq \frac{C_t\omega_0}{|s_D|^{2(1-\beta)}_{h_D}}.
$$
Therefore $$diam(\omega_{t,\epsilon})\leq C_t. $$
\begin{prop}
$(M,\omega_{t})$ is the Gromov-Hausdorff limit of $(M,\omega_{t,\epsilon})$ as $\epsilon\rightarrow 0$.
\end{prop}
The proof of the above proposition is the same as the proposition(2.5) of \cite{CDS1}, so we omit it.
\section{A prior estimate to the conical continuity equation}
In this section, we present some estimate to the conical continuity equation (1.1). First, we assume $\beta\in \mathbb{Q}$ and $L_D$ is a semi-positive line bundle. The rationality theorem of Kawamata \cite{Ka} says that $T\in \mathbb{Q}$. Take a positive integer $l_0$ such that $Tl_0\in \mathbb{Z}, Tl_0(1-\beta)\in \mathbb{Z}$ and define the limit line bundle $L=l_0(L^{'}+TK_{M}+T(1-\beta)L_D)$.

Since the limit class $L^{'}+TK_{M}+T(1-\beta)L_D$ is nef and big, according to the base point free theorem \cite{KoMo},
we may assume $l_0$ is chosen such that $L$ has no base points. A basis of $H^0(M,L)$ gives a holomorphic map
$$
  \Phi:M\longrightarrow \mathbb{C}P^N
$$
where $N=dimH^0(M,L)-1$. Let $M_{reg}$ be the set of regular points of $\Phi$. Denote by $\omega_{FS}$ the Fubini-Study metric of $\mathbb{C}P^N$ and $\eta_{T}=\frac{1}{l_0}\Phi^{*}\omega_{FS}$.

Let $\omega_{t,\epsilon}$, $t\in [T-\bar{t},T)$, be a solution to (2.7). By putting $\eta_t=\frac{T-t}{T}\omega_0+\frac{t}{T}\eta_T$, a family of background metrics, the solution $\omega_{t,\epsilon}$ can
be written as
$$
 \omega_{t,\epsilon}=\eta_t+\sqrt{-1}\partial\overline{\partial}u_{t,\epsilon}.
$$
Since $\frac{1}{T}(\omega_0-\eta_{T})\in c_1(M)-(1-\beta)c_1(L_D)$, there is a smooth volume form $\Omega$ on $M$ and curvature $\Theta_{h_D}$ on $L_D$ such that$$Ric(\Omega)-(1-\beta)\Theta_{h_D}=\frac{1}{T}(\omega_0-\eta_{T}).$$

Now we consider the following equation
\begin{equation}
(\eta_t+\sqrt{-1}\partial\overline{\partial}u_{t,\epsilon})^n=e^{\frac{u_{t,\epsilon}}{t}}\frac{\Omega}{(\epsilon^2+|s_D|^2_{h_D})^{1-\beta}}
\end{equation}
\begin{lem}
There is a constant $C$ independent of $t$ and $\epsilon$ such that
$$  |u_{t,\epsilon}|_{C^0}\leq C.$$
\end{lem}
\begin{proof}
The uniform upper bound of $u_{t,\epsilon}$ is trivial consequence of the maximum principle. The $L^{\infty}$ bound follows from the capacity calculation of \cite{Zh} for exactly our case when $u_{t,\epsilon}$ has a uniform upper bound.
\end{proof}
\begin{cor}
There exists $C$ independent of $t$ and $\epsilon$ such that
$$C^{-1}\Omega\leq \omega_{t,\epsilon}^{n} \leq \frac{C\Omega}{|s_D|_{h_D}^{2(1-\beta)}}.                       $$
\end{cor}
\begin{lem}
There exists $C$ independent of $t$ and $\epsilon$ such that
$$\dot{u_{t,\epsilon}}\leq C; \ddot{u_{t,\epsilon}}\leq C.$$
\end{lem}
\begin{proof}
Differentiating $t$ at both sides of (3.1), one gets
$$
tr_{\omega_{t,\epsilon}}\omega_{t,\epsilon}^{'}=\frac{1}{t^2}(t\dot{u_{t,\epsilon}}-u_{t,\epsilon})
$$
where
$$
\omega_{t,\epsilon}^{'}=\frac{1}{T}(\eta_T-\omega_0)+\sqrt{-1}\partial\overline{\partial}\dot{u_{t,\epsilon}}=
\frac{1}{t}(\omega_{t,\epsilon}-\omega_0-\sqrt{-1}\partial\overline{\partial}u_{t,\epsilon})
+\sqrt{-1}\partial\overline{\partial}\dot{u_{t,\epsilon}}.
$$
By the simple calculation one has
$$
\triangle_{\omega_{t,\epsilon}}(t\dot{u_{t,\epsilon}}-u_{t,\epsilon})=\frac{1}{t}(t\dot{u_{t,\epsilon}}-u_{t,\epsilon})-n+tr_{\omega_{t,\epsilon}}\omega_0.
$$
Applying the maximum principle one derives
$$
  t\dot{u_{t,\epsilon}}-u_{t,\epsilon}\leq C.
$$
Combining with the $C^0$ bound of $u_{t,\epsilon}$ we also have
$$
 \dot{u_{t,\epsilon}}\leq C.
$$
To get the upper bound of $\ddot{u_{t,\epsilon}}$ we first observe that
$$
 t\dot{u_{t,\epsilon}}-u_{t,\epsilon}=t^2tr_{\omega_{t,\epsilon}}\omega_{t,\epsilon}^{'}.
$$
Differentiating $t$ at both sides of the above formula one gets
$$
t\ddot{u_{t,\epsilon}}=2t \cdot tr_{\omega_{t,\epsilon}}\omega_{t,\epsilon}^{'}+t^2\triangle_{\omega_{t,\epsilon}}\ddot{u_{t,\epsilon}}
-t^2|\omega_{t,\epsilon}^{'}|^2=t^2\triangle_{\omega_{t,\epsilon}}\ddot{u_{t,\epsilon}}
-|\omega_{t,\epsilon}-t\omega_{t,\epsilon}^{'}|^2+n.
$$
Then by the maximum principle one derives
$$
\ddot{u_{t,\epsilon}}\leq C.
$$
\end{proof}
By theorem(1.2), one knows that $u_{t,\epsilon}$ $C^{\infty}$ converges to $u_t$ on each compact subset $K\subset M\setminus D$ when $\epsilon\rightarrow 0$. Furthermore $u_t$ solves the following equation in the current sense
$$
(\eta_t+\sqrt{-1}\partial\overline{\partial}u_{t})^n=e^{\frac{u_{t}}{t}}\frac{\Omega}{|s_D|^{2(1-\beta)}_{h_D}}.
$$
\begin{lem}
The function $u_t$ converges uniformly to a bounded function $u_T$ satisfying
$$
(\eta_T+\sqrt{-1}\partial\overline{\partial}u_{T})^n=e^{\frac{u_{T}}{T}}\frac{\Omega}{|s_D|^{2(1-\beta)}_{h_D}}
$$
in the current sense.
\end{lem}
\begin{proof}
For $u_t$ we observe that
$$
t\triangle_{\omega_t}(\frac{u_t}{t})^{'}=(\frac{u_t}{t})'+\frac{1}{t}(tr_{\omega_t}\omega_0-n)\geq (\frac{u_t}{t})'-\frac{n}{t}.
$$
Then one deduces
$$
 t\triangle_{\omega_t}(\frac{u_t}{t}-n\log t)'\geq (\frac{u_t}{t}-n\log t)^{'}.
$$
By the maximum principle one knows that $\frac{u_t}{t}-n\log t$ is monotone decreasing. Consequently, $u_t$ converges uniformly to a unique limit $u_T$. It is obvious that $u_T$ is smooth outside $M\backslash (\mathcal{S}_{M}\cup D)$.
\end{proof}
\begin{prop}
There exists $C$ independent of $t$ and $\epsilon$ such that
$$
  \eta_T\leq C\omega_{t,\epsilon}, \forall t\in [T-\bar{t}, T).
$$
\end{prop}
\begin{proof}
By Yau's Schwarz lemma \cite{Y} and $Ric(\omega_{t,\epsilon})\geq -\frac{1}{t}\omega_{t,\epsilon}$,
$$
 \triangle_{\omega_{t,\epsilon}}\log tr_{\omega_{t,\epsilon}}\eta_{T}\geq -\frac{n}{t}-ntr_{\omega_{t,\epsilon}\eta_T}.
$$
On the other hand, $\eta_t\geq \delta \eta_T$ for some $\delta>0$ independent of $t$, so
$$
\triangle_{\omega_{t,\epsilon}}u_{t,\epsilon}=n-tr_{\omega_{t,\epsilon}}\eta_{t}\leq n-\delta tr_{\omega_{t,\epsilon}}\eta_{T}.
$$
Hence
$$
\triangle_{\omega_{t,\epsilon}}(\log tr_{\omega_{t,\epsilon}}\eta_{T}-\frac{2n}{\delta}u_{t,\epsilon})\geq
n tr_{\omega_{t,\epsilon}}\eta_{T}-\frac{C(n,T)}{\delta}.
$$
Let $H=\log tr_{\omega_{t,\epsilon}}\eta_{T}-\frac{2n}{\delta}u_{t,\epsilon}$. Assume $H$ achieves maximum at $x_0$, then
$$
  tr_{\omega_{t,\epsilon}}\eta_{T}(x_0)\leq C.
$$
By the boundness of $u_{t,\epsilon}$, one has
$$
  tr_{\omega_{t,\epsilon}}\eta_{T}\leq C.
$$
\end{proof}
\begin{cor}
The limit metric $\omega_{T}$ is smooth on $M_{reg}\backslash D$.
\end{cor}
\begin{proof}
$\eta_T$ is smooth on any compact subset $K\subset M_{reg}\backslash D$, so by Lemma(3.5) and proposition(3.7) one knows
$$
C^{-1}\eta_T\leq \omega_T\leq C_K\eta_T.
$$
In particular, $n+\triangle_{\eta_T}u_T\leq C_K$ on $K$. Then applying a bootstrap argument we get the higher derivative bound $|u_T|_{C^{l}(K)}\leq C_{l,K}$.
\end{proof}

Now we define $w_{t,\epsilon}=(T-t)\dot{u_{t,\epsilon}}+u_{t,\epsilon}$ which satisfies
\begin{equation}
 \triangle_{\omega_{t,\epsilon}}w_{t,\epsilon}=\frac{1}{t}w_{t,\epsilon}-\frac{T}{t^2}u_{t,\epsilon}+n
 -tr_{\omega_{t,\epsilon}}\eta_T.
\end{equation}
This can be seen by combining
$$
\triangle_{\omega_{t,\epsilon}}\dot{u_{t,\epsilon}}=\frac{1}{t^2}(t\dot{u_{t,\epsilon}}-u_{t,\epsilon})
+\frac{1}{T}tr_{\omega_{t,\epsilon}}(\omega_0-\eta_T)
$$
and
$$
\triangle_{\omega_{t,\epsilon}}u_{t,\epsilon}=n-\frac{T-t}{T}tr_{\omega_{t,\epsilon}}\omega_0-\frac{t}{T}tr_{\omega_{t,\epsilon}}\eta_T.
$$
For (3.8) by maximum principle one gets $$w_{t,\epsilon}\geq -c$$. Therefore $$|w_{t,\epsilon}|_{C^0}\leq C,
|\triangle_{\omega_{t,\epsilon}}w_{t,\epsilon}|_{C^0}\leq C.$$
Combining with the $C^0$ bound of $u_{t,\epsilon}$ we also have
$$
-\frac{C}{T-t}\leq \dot{u_{t,\epsilon}}\leq C, \ \forall t\in[T-\bar{t},T).
$$
\begin{prop}
There exists $C$ independent of $t$ and $\epsilon$ such that
$$
 |\nabla w_{t,\epsilon}|_{C^0}\leq C, \ \forall t\in[T-\bar{t},T).
$$
In particular, since $\dot{u_{t,\epsilon}}$ converges to a locally bounded function on $M\backslash (\mathcal{S}_{M}\cup D)$ as $t\rightarrow T$ and $\epsilon\rightarrow 0$, one has
$$
 |\nabla u_T|_{C^0}\leq C, \forall t\in[T-\bar{t},T).
$$
\end{prop}
\begin{proof}
Recall that $Ric(\omega_{t,\epsilon})\geq -\frac{1}{t}(\omega_{t,\epsilon}-\omega_0)$, so by the Bochner formula,
$$
\triangle|\nabla w_{t,\epsilon}|^2\geq |\nabla\nabla w_{t,\epsilon}|^2+|\nabla\overline{\nabla} w_{t,\epsilon}|^2
+\nabla_i\triangle w_{t,\epsilon}\cdot \nabla_{\bar{i}}w_{t,\epsilon}+\nabla_{\bar{i}}\triangle w_{t,\epsilon}\cdot
\nabla_i w_{t,\epsilon}+\frac{1}{t}(\omega_0-\omega_{t,\epsilon})_{i\bar{j}}\nabla_i w_{t,\epsilon} \nabla_{\bar{j}}
w_{t,\epsilon}
$$
where we omit the metric for the convenience. By (3.8) one has
$$
\nabla_i\triangle w_{t,\epsilon}\cdot \nabla_{\bar{i}}w_{t,\epsilon}+\nabla_{\bar{i}}\triangle w_{t,\epsilon}\cdot
\nabla_i w_{t,\epsilon}=\frac{2}{t}|\nabla w_{t,\epsilon}|^2-2Re(\nabla_i tr_{\omega_{t,\epsilon}\eta_T}\cdot \nabla_{\bar{i}} w_{t,\epsilon})-\frac{2T}{t^2}Re(\nabla_i u_{t,\epsilon}\cdot \nabla_{\bar{i}}w_{t,\epsilon}).
$$
So,
$$
\triangle|\nabla w_{t,\epsilon}|^2\geq \frac{1}{2t}|\nabla w_{t,\epsilon}|^2-4t|\nabla tr_{\omega_{t,\epsilon}}\eta_T|^2
-\frac{16T^2}{t^3}|\nabla u_{t,\epsilon}|^2.
$$
Notice that
\begin{align*}
\triangle tr_{\omega_{t,\epsilon}}\eta_T &\geq tr_{\omega_{t,\epsilon}}\eta_T(-\frac{n}{t}-Atr_{\omega_{t,\epsilon}}\eta_T)+\frac{1}{tr_{\omega_{t,\epsilon}}\eta_T}
|\nabla tr_{\omega_{t,\epsilon}}\eta_T|^2 \\
&\geq -C+C|\nabla tr_{\omega_{t,\epsilon}}\eta_T|^2,
\end{align*}
and
$$
\triangle(-u_{t,\epsilon})=-n+\frac{T-t}{T}tr_{\omega_{t,\epsilon}}\omega_0+\frac{t}{T}tr_{\omega_{t,\epsilon}}\eta_T \\
\geq \frac{T-t}{T}tr_{\omega_{t,\epsilon}}\omega_0-C.
$$
and
$$
\triangle u_{t,\epsilon}^{2}=2u_{t,\epsilon}\triangle u_{t,\epsilon}+2|\nabla u_{t,\epsilon}|^2\geq 2|\nabla u_{t,\epsilon}|^2-C\frac{T-t}{T}tr_{\omega_{t,\epsilon}}\omega_0-C.
$$
Let $H=|\nabla w_{t,\epsilon}|^2+4tCtr_{\omega_{t,\epsilon}}\eta_T+\frac{8T^2}{t^3}u_{t,\epsilon}^{2}- \frac{8T^2}{t^3}Cu_{t,\epsilon}$, then one obtains
$$
 \triangle H\geq \frac{1}{2t}|\nabla w_{t,\epsilon}|^2-C,
$$
by the maximum principle one gets
$$
|\nabla w_{t,\epsilon}|_{C^0}\leq C.
$$
\end{proof}

\section{Algebraic structure of the limit space}
\subsection{Preliminaries}
In this subsection we introduce some useful formulas on a general line bundle. Let $(M,\omega)$ be a K\"{a}hler manifold of dimension $n$ and $(L,h)$ be a Hermitian line bundle over M. Let $\Theta_{h}$ be the Chern curvature form of $h$. Let $\nabla$ and $\overline{\nabla}$ denote the $(1, 0)$ and $(0, 1)$ part of
a connection respectively. The connection appeared in this paper is usually known as the Chern connection or Levi-Civita connection.

For a holomorphic section $\tau \in H^0(M,L)$ we write for simplicity
$$
 |\tau|=|\tau|_{h}, \ |\nabla \tau|_{h\otimes \omega}=|\nabla \tau|,
$$
and
$$
|\nabla\nabla \tau|^2=\sum_{i,j}|\nabla_{i}\nabla_{j}\tau|^2, \ |\nabla\overline{\nabla} \tau|^2=\sum_{i,j}|\nabla_{i}\nabla_{\bar{j}}\tau|^2.
$$
By direct computation we have
\begin{lem}(Bochner formulas).
For any $\tau \in H^0(M,L)$ one has
\begin{equation}
 \triangle_{\omega}|\tau|^2=|\nabla \tau|^2-|\tau|^2\cdot tr_{\omega}\Theta_{h}
\end{equation}
and
\begin{multline}
 \triangle_{\omega}|\nabla \tau|^2=|\nabla\nabla \tau|^2+|\nabla\overline{\nabla} \tau|^2   -\nabla_j(\Theta_h)_{i\bar{j}} \langle \tau,\nabla_{\bar{i}}\bar{\tau}\rangle-\nabla_{\bar{j}}(tr_{\omega}\Theta_h)\langle \nabla_j \tau,\bar{\tau}\rangle  \\  +R_{i\bar{j}}\langle \nabla_j \tau,\nabla_{\bar{i}}\bar{\tau}\rangle
 -2(\Theta_h)_{i\bar{j}}\langle\nabla_j\tau,\nabla_{\bar{i}}\bar{\tau}\rangle-|\nabla \tau|^2\cdot tr_{\omega}\Theta_h
\end{multline}
where $R_{i\bar{j}}$ is the Ricci curvature of $\omega$, $\langle,\rangle$ is the inner product defined by $h$.
\end{lem}
\subsection{Gromov-Hausdorff convergence: global convergence}
In this subsection we consider a family of manifolds $(M, \omega_{t, \epsilon})$ on which the lower bound of  Ricci curvature can be controlled, i.e. $Ric(\omega_{t, \epsilon})\geq -\frac{1}{T-\bar{t}}\omega_{t, \epsilon}$ for
$t\in[T-\bar{t},T)$. By Gromov precompactness theorem, passing to a subsequence $(t_i, \epsilon_i) \rightarrow (T,0)$
and fix $x_0\in M\backslash (\mathcal{S}_{M}\cup D)$, we may assume that
$$
  (M, \omega_{t_i,\epsilon_i}, x_0)\xrightarrow{d_{GH}} (M_T, d_T, x_T).
$$
The limit $(M_T,d_T)$ is a complete length metric space, maybe noncompact in a prior. It has a regular/singular decomposition $M_T=\mathcal{R}\cup \mathcal{S}$, a point $x\in \mathcal{R}$ iff the tangent cone at $x$ is the Euclidean space $\mathbb{R}^{2n}$. The proof of the following lemma is exactly same as \cite{GS} so we omit it.
\begin{lem}
There is a sufficiently small constant $\delta>0$ such that for any $t\in[T-\bar{t},T)$ and $\epsilon\geq 0$, if a metric ball $B_{\omega_{t,\epsilon}}(x,r)$ satisfies
$$ Vol(B_{\omega_{t,\epsilon}}(x,r))\geq (1-\delta)Vol(B^0_{r}) \ and \ B_{\omega_{t,\epsilon}}(x,r)\cap D=\varnothing$$ where $Vol(B^0_{r})$ is the volume of a metric ball of radius $r$ in $2n$-Euclidean space,
then $$Ric(\omega_{t,\epsilon})\leq (2n-1)r^{-2}\omega_{t,\epsilon}, \ in \ B_{\omega_{t,\epsilon}}(x, \delta r).$$
\end{lem}
\begin{lem}
The regular set $\mathcal{R}$ is open in the limit space $(M_T, d_T, x_T)$.
\end{lem}
\begin{proof}
We follow Tian's argument \cite{T1}. By Proposition (2.9), one has $(M, \omega_{t_i}, x_0)\xrightarrow{d_{GH}} (M_T, d_T, x_T)$. If $x\in \mathcal{R}$, then by Colding's volume convergence theorem \cite{Co}, there exists $r=r(x)>0$ such that $\mathcal{H}^{2n}(B_{d_T}(x,r))\geq (1-\frac{\delta}{2})Vol(B^0_{r})$, where $\mathcal{H}^{2n}$ denotes the Hausdorff measure. Let $\{x_i\}$ be a sequence of points in $M$ such that $x_i\xrightarrow{d_{GH}} x$, then by the volume convergence theorem again, $Vol(B_{\omega_{t_i}}(x_i,r))\geq (1-\delta)Vol(B^0_{r})$ for $i$ sufficiently large. On the other hand, if $y_i\in D$, then by the
Bishop-Gromov volume comparison theorem, for any $\bar{r}>0$, one has (set $a=-\frac{1}{T-\bar{t}}$)
$$
  \frac{Vol(B_{\omega_{t_i}}(y_i,\bar{r}))}{Vol(B_{\bar{r}}^{a})}\leq \beta.
$$
Furthermore, when $\bar{r}$ is sufficiently small, we have
$$
\frac{Vol(B_{\omega_{t_i}}(y_i,\bar{r}))}{Vol(B_{\bar{r}}^{0})}=\frac{Vol(B_{\omega_{t_i}}(y_i,\bar{r}))}{Vol(B_{\bar{r}}^{a})}\cdot \frac{Vol(B_{\bar{r}}^{a})}{Vol(B_{\bar{r}}^{0})}\leq (1+\delta)\beta.
$$
Note that $B_{\omega_{t_i}}(x_i,r)\xrightarrow{d_{GH}}B_{d_T}(x,r)$, so by the
Bishop-Gromov volume comparison theorem there exists an $N=N(\delta)$ such that for any $\tilde{r}\in (0,\frac{r}{N})$ and
$y_i\in B_{\omega_{t_i}}(x_i,\tilde{r})$, one gets
$$
1-\delta\leq \frac{Vol(B_{\omega_{t_i}}(y_i,\tilde{r}))}{Vol(B_{\omega_{t_i}}(x_i,\tilde{r}))}\leq 1+\delta.
$$
Now, we claim that $B_{\omega_{t_i}}(x_i,r')\cap D=\varnothing$ where $r'=min\{\bar{r},\tilde{r}\}$. If this claim is false, we assume $y_i\in B_{\omega_{t_i}}(x_i,r')$ for all i sufficiently large, we have
$$
1-\delta\leq \frac{Vol(B_{\omega_{t_i}}(x_i,r'))}{Vol(B_{r'}^{0})}\leq (1+\delta)\frac{Vol(B_{\omega_{t_i}}(y_i,r'))}{Vol(B_{r'}^{0})}\leq (1+\delta)^2\beta.
$$
Then we get a contradiction if $\delta$ is chosen sufficiently small. According to above lemma, together with Anderson¨s harmonic radius estimate \cite{An}, there is $\delta^{'}=\delta^{'}(\alpha)>0$ for any $0<\alpha<1$ such that the $C^{1,\alpha}$ harmonic radius at $x_i$ is bigger than $\delta^{'} r'$. Passing to the limit, it gives a harmonic coordinate on $B_{d_T}(x,\delta^{'} r')$. This implies in particular that $B_{d_T}(x,\delta^{'}r')\subset \mathcal{R}$. So $\mathcal{R}$ is open with a $C^{1,\alpha}$ K\"{a}hler metric, denoted by $\overline{\omega_T}$; moreover the metric $\omega_{t_i,\epsilon_i}$ or $\omega_{t_i}$ converges in $C^{1,\alpha}$ topology to $\overline{\omega_T}$ on $\mathcal{R}$ for any $0<\alpha<1$.
\end{proof}

For any metric $\omega$, let $d_\omega$ be the length metric induced by $\omega$.
\begin{lem}
$(M_T,d_T)=\overline{(\mathcal{R}, d_{\overline{\omega_T}})}$, the metric completion of $(\mathcal{R},d_{\overline{\omega_T}})$.
\end{lem}
\begin{proof}
By the previous argument one finds an exhaustion of $\mathcal{R}$ by compact subsets $K_i$ with $K_i\subset K_{i+1}$ and a sequence of embeddings
$\phi_i: K_i\rightarrow M$ such that $\phi_i(x_T)=x_0$. Thus $\phi_i$ defines a Gromov-Hausdorff approximation of the convergence $(M, \omega_{t_i,\epsilon_i}, x_0)\xrightarrow{d_{GH}} (M_T, d_T, x_T)$ because $Codim(\mathcal{S})\geq 2$ \cite{ChCo2}. There is a fact that $\phi_{i}^{*}\omega_{t_i,\epsilon_i}\xrightarrow{C^{1,\alpha}}\overline{\omega_T}$ which demonstrates that $(\mathcal{R},d_T|_{\mathcal{R}})=(\mathcal{R},d_{\overline{\omega_T}})$. Notice that $(\mathcal{R},d_T)$ is dense in $(M_T, d_T, x_T)$ because $Codim(\mathcal{S})\geq 2$. Therefore the lemma is proved.
\end{proof}
\begin{lem}
$\mathcal{R}$ is geodesically convex in $M_T$ in the sense that any minimal geodesic with endpoints in $\mathcal{R}$ lies in $\mathcal{R}$.
\end{lem}
\begin{proof}
It is simply a consequence of Colding-Naber¨s H\"{o}lder continuity of tangent cones along a geodesic in $M_T$ \cite{CN}. Actually, in \cite{ChCo2} one knows
that any pair of regular points can be connected by a curve consisting entirely of almost regular points and knows $\mathcal{R}$ is locally convex by previous argument. Therefore, the tangent cone of each point which is in a minimal geodesic connecting any pair of regular points is $\mathbb{R}^{2n}$.
\end{proof}

Let $D^{'}$ be any divisor such that $D \cup \mathcal{S}_{M} \subset D^{'}$. Define the Gromov-Hausdorff limit of $D^{'}$
$$
 D_{T}^{'}:=\{x\in M_T|there \ exists \ x_i\in D^{'} such \ that \ x_i\xrightarrow{d_{GH}} x\}.
$$
\begin{prop}
$(M_T,d_T)$ is isometric to $\overline{(M\backslash D^{'},d_{\omega_T})}$.
\end{prop}
\begin{proof}
First, by the argument of \cite{RZ} one knows that $(M_T\backslash D_T^{'},\overline{\omega_T})$ is isometric to $(M\backslash D^{'},\omega_T)$; moreover $M_T\backslash D_T^{'}\subset \mathcal{R}$. We make the following
\begin{claim}
$D_T^{'}\backslash \mathcal{S}$ is a subvariety of dimension $(n-1)$ if it is not empty.
\end{claim}
\begin{proof}
Let $x\in D_T^{'}\backslash \mathcal{S}$ and $x_i\in D^{'}$ such that $ x_i\xrightarrow{d_{GH}} x$. By the $C^{1,\alpha}$ convergence of $\omega_{t_i,\epsilon_i}$ around $x$, there are $C, r>0$ independent of $i$ and a sequence of harmonic coordinates in $B_{\omega_{t_i,\epsilon_i}}(x_i,r)$
such that $C^{-1}\omega_E\leq \omega_{t_i,\epsilon_i}\leq C\omega_E$ where $\omega_E$ is the Euclidean metric in the coordinates. Since the total volume of
$D^{'}$ is uniformly bounded for any $\omega_{t_i,\epsilon_i}$, the local analytic $D^{'}\cap B_{\omega_{t_i,\epsilon_i}}(x_i,r)$ have a uniform bound of degree and so converge to an analytic set $D_T^{'} \cap B_{d_T}(x,r)$.
\end{proof}
From the above Claim we know that $dim(D_T^{'})=dim(\mathcal{S}\cup (D_T^{'}\backslash \mathcal{S}))\leq 2n-2$. Thus by the argument of \cite{ChCo2}, one can show that the length metric $d_{\overline{\omega_T}}$ on $M_T\backslash D_T^{'}$ is the same as $d_T$. Therefore
$$
 (M_T,d_T)=\overline{(M_T\backslash D_T^{'},d_{\overline{\omega_T}})}=\overline{(M\backslash D^{'},d_{\omega_T})}.
$$
\end{proof}
Combining with Proposition(2.8), a direct corollary is
\begin{cor}
$(M,\omega_t, x_0)$ converges globally to $(M_T,d_T, x_T)$ under the Gromov-Hausdorff topology as $t\rightarrow T$.
\end{cor}

Let $M_{sing}$ be the subvariety of critical points of $\Phi$ which is defined in section $3$ and $M_{reg}=M\backslash M_{sing}$. We have
shown that $\omega_T$ is a smooth metric on $M_{reg}\backslash D$. Another corollary is
\begin{cor}
$(M_T,d_T)$ is isometric to $\overline{(M_{reg}\backslash D,d_{\omega_T})}$.
\end{cor}
\begin{proof}
We choose a divisor $D^{'}$ such that $D^{'}\supset (D\cup \mathcal{S}_M\cup M_{sing})$, then $M\backslash D^{'}\subset M_{reg}\backslash D$.
Notice that $(M_{reg}\backslash D)\backslash  (M\backslash D^{'})=(M_{reg}\backslash D)\cap D^{'}$ has real codimension larger than $2$ in
$(M_{reg}\backslash D,\omega_T)$. Thus the length metric $d_{\omega_T}$ on $M\backslash D^{'}$ equals to the restricted extrinsic metric from
$(M_{reg}\backslash D,\omega_T)$. Since $M\backslash D^{'}$ is dense in $M_{reg}\backslash D$, we conclude
$$
 (M_T,d_T)=\overline{(M\backslash D^{'},d_{\omega_T})}=\overline{(M_{reg}\backslash D,d_{\omega_T})}.
$$
\end{proof}
\begin{lem}
The identity map $\id: \ M_{reg}\backslash D\rightarrow M$ gives a Gromov-Hausdorff approximation representing the convergence $(M,\omega_t, x_0)\rightarrow (M_T,d_T,x_T)$ as $t\rightarrow T$.
\end{lem}
\begin{proof}
First we observe that $(M\backslash D^{'},d_T)=(M\backslash D^{'},d_{\omega_T})$ and $(M\backslash D^{'},d_T)$ is dense in$(M_{reg}\backslash D,d_T)$.
Thus $\id: \ (M\backslash D^{'}, d_{\omega_T})\rightarrow (M,\omega_t)$ defines a Gromov-Hausdorff approximation because $(M\backslash D^{'}, d_{\omega_T})$
is dense in $(M_T,d_T)$.
\end{proof}

Therefore, the identity map $\id$ extends to an isometry
$$
 \overline{\id}: \ \overline{(M_{reg}\backslash D,d_{\omega_T})}\rightarrow (M_T,d_T).
$$
Since $\omega_T$ is smooth on $M_{reg}\backslash D$, one sees that $M_{reg}\backslash D\subset \mathcal{R}$.
\begin{prop}
\begin{enumerate}
\item  $\omega_{t,\epsilon}$ converges smoothly to $\omega_T$ on $M_{reg}\backslash D$ as $t\rightarrow T$ and $\epsilon \rightarrow 0$. \\
\item  $\overline{\id}(M_{reg}\backslash D)=\mathcal{R}$, the regular set of $M_T$.
\end{enumerate}
\end{prop}
\begin{proof}
(1) For any compact subset $K\subset M_{reg}\backslash D\subset \mathcal{R}$, there exists $r=r_K>0$ such that $Vol(B_{d_T}(x,r))\geq (1-\frac{\delta}{2})Vol(B^0_{r})$ for any $x\in K$. where $\delta$ is the constant in Lemma (4.4). Then, since the identity map represents the
Gromov-Hausdorff convergence, we have $Vol(B_{\omega_{t,\epsilon}}(x,r))\geq (1-\delta) Vol(B^0_{r})$ for any $x\in K$, $t$ sufficiently close to $T$ and
$\epsilon$ sufficiently close to $0$. By Lemma(4.4), the Ricci curvature $Ric(\omega_{t,\epsilon})\leq C \omega_{t,\epsilon}$ uniformly on $K$ for some constant
$C=C(K)$. Since
$$
\omega_{t,\epsilon}= \omega_0-tRic(\omega_{t,\epsilon})+(1-\beta)\frac{\epsilon^2\langle\nabla s,\overline{\nabla s}\rangle}{(\epsilon^2+|s_D|^2_{h_D})^2}+(1-\beta)\frac{\epsilon^2}{\epsilon^2+|s_D|^2_{h_D}}\Theta_{h_D},
$$
one sees that $\omega_{t,\epsilon}\geq C^{-1}\omega_0$. Notice that
$$
  tr_{\omega_0}\omega_{t,\epsilon}\leq (tr_{\omega_{t,\epsilon}}\omega_{0})^{n-1}\frac{(\omega_{t,\epsilon})^n}{\omega_0^{n}}.
$$
Together with the uniform $L^{\infty}$ bound of $u_{t,\epsilon}$, one gets
$$
  C^{-1}\omega_0\leq \omega_{t,\epsilon}\leq C\omega_0, \ on \ K.
$$
Then by a standard bootstrap argument, we prove that $\omega_{t,\epsilon}$ converges smoothly to $\omega_T$ on $K$. \\

(2) We only to prove $M_{reg}\backslash D\supset \mathcal{R}$. We argue by contradiction. Suppose there is a point $p\in \mathcal{R}\backslash (M_{reg}\backslash D)$, then there exists a family of points $p_i\in M_{sing}\cup D$ such that $p_i\xrightarrow{d_{GH}} p$. We will divide the discussion into two parts.

On one hand, if there exists $p_i\xrightarrow{d_{GH}} p$ for each $p_i\in D$, that is a contradiction by Lemma (4.5).

On the other hand, there exists $p_i\xrightarrow{d_{GH}} p$ for each  $p_i\in M_{sing}\backslash D$. By $C^{1,\alpha}$ convergence on $\mathcal{R}$, there exist
$C,r>0$ independent of $t$ and $\epsilon$ and a sequence of harmonic coordinates on $B_{\omega_{t_i,\epsilon_i}}(p_i,r)$ such that $C^{-1}\omega_E\leq \omega_{t_i,\epsilon_i}\leq C\omega_E$ where $\omega_E$ is the Euclidean metric in this coordinate. Denote $m=dim_{\mathbb{C}}M_{sing}$. Then
$$
 Vol_{\omega_{t_i,\epsilon_i}}((M_{sing}\backslash D)\cap B_{\omega_{t_i,\epsilon_i}}(p_i,r))=\int_{(M_{sing}\backslash D)\cap B_{\omega_{t_i,\epsilon_i}}(p_i,r)}\omega_{t_i,\epsilon_i}^{m}\geq \int_{(M_{sing}\backslash D)\cap B_{\omega_E}(C^{-\frac{1}{2}})}(C^{-1}\omega_E)^m
$$
which has a uniform lower bound $C^{-2m}c(m)r^{2m}$ where $c(m)$ is the volume of unit sphere in $\mathbb{C}^m$. However, this contradicts with the
degeneration of the limit metric $\eta_T$ along $M_{sing}$:
$$
Vol_{\omega_{t_i,\epsilon_i}}((M_{sing}\backslash D)\cap B_{\omega_{t_i,\epsilon_i}}(p_i,r))\leq Vol_{\omega_{t_i,\epsilon_i}}(M_{sing}\backslash D)=
\int_{M_{sing}\backslash D}\omega_{t_i,\epsilon_i}^{n}=\left(\frac{T-t_i}{T} \right)^{m}\int_{M_{sing}\backslash D}\omega_0^{m}
$$
which tends to $0$ as $t_i\rightarrow T$. So we have $M_{reg}\backslash D\supset \mathcal{R}$.
\end{proof}

\subsection{$L^{\infty}$estimate to holomorphic sections}
Let $L=l_0(L^{'}+TK_{M}+T(1-\beta)L_{D})$ be the limit line bundle. Choose a Hermitian metric $h_{L^{'}}$ on $L^{'}$ whose curvature form $\Theta_h=\omega_0$ and put $h_{t,\epsilon}=h^{l_0}_{L^{'}}\otimes (\omega_{t,\epsilon}^{-n})^{l_0T}\otimes h_{D}^{l_0T(1-\beta)}\cdot e^{-l_0T(1-\beta)\log(\epsilon^2+|s_D|^2_{h_D})}$, a family of Hermitian metric on $L$ for any $t\in [T-\bar{t},T)$. The curvature form of $h_{t,\epsilon}$ is
$$
 \Theta_{h_{t,\epsilon}}=l_0\frac{T}{t}\omega_{t,\epsilon}-l_0\frac{T-t}{t}\omega_0\leq l_0\frac{T}{t}\omega_{t,\epsilon} .
$$
So, by the Bochner formula (4.2) we have
$$
\triangle_{\omega_{t,\epsilon}}|\tau|^2_{h_{t,\epsilon}^k}=|\nabla \tau|^{2}_{h_{t,\epsilon}^k}-knl_0\frac{T}{t}|\tau|^{2}_{h_{t,\epsilon}^k}, \ \forall \
\tau \in H^0(M,L^k).
$$
Also recall that we have the following well-known Sobolev inequality: for any $R>0$, there is $C(R)$ independent of $t$ and $\epsilon$ such that
$$
  \left(\int_{B_{\omega_{t,\epsilon}}(x_0,R)}|f|^{\frac{2n}{n-1}}\omega_{t,\epsilon}^{n}\right)^{\frac{n-1}{n}}\leq C(R)\int_{B_{\omega_{t,\epsilon}}(x_0,R)}
  |f|^2+|\nabla f|^2_{\omega_{t,\epsilon}}\omega_{t,\epsilon}^{n}.
$$
for all $f\in C^1_{0}(B_{\omega_{t,\epsilon}}(x_0,R))$.

By a standard iteration argument(Lemma 3.14 \cite{NTZ}) we have
\begin{lem}
For any $R>0$, there exists $C(R)$ independent of $t$, $\epsilon$ and $k\geq 1$ such that for any $t\in [T-\bar{t},T)$ and $B_{\omega_{t,\epsilon}}(x,2r)\subset B_{\omega_{t,\epsilon}}(x_0,R)$, if $\tau \in H^0(B_{\omega_{t,\epsilon}}(x,2r),L^k)$, then
$$
 \sup_{B_{\omega_{t,\epsilon}}(x,r)}|\tau|^2_{h^k_{t,\epsilon}}\leq C(R)\cdot r^{-2n}\cdot k^n\cdot\int_{B_{\omega_{t,\epsilon}}(x,2r)}|\tau|^2_{h^k_{t,\epsilon}}\omega_{t,\epsilon}^{n}.
$$
\end{lem}

Recall the Gromov-Hausdorff convergence
$$
 (M,\omega_{t,\epsilon},x_0)\xrightarrow{d_{GH}} (M_T,d_T,x_T).
$$
Define the Hermitian line bundle $(L_T,h_T)$ on the regular set $\mathcal{R}\subset M_T$ by
$$
L=l_0(L^{'}+TK_{\mathcal{R}}+T(1-\beta)L_{D}),\ h_T=h^{l_0}_{L^{'}}\otimes (\omega_{T}^{-n})^{l_0T}\otimes h_{D}^{l_0T(1-\beta)}\cdot e^{-l_0T(1-\beta)\log|s_D|^2_{h_D}}.
$$
Under the isometry $\overline{\id}: \ \overline{(M_{reg}\backslash D,d_{\omega_T})}\rightarrow (M_T,d_T)$ and $\mathcal{R}=M_{reg}\backslash D$, we know that
the Hermitian line bundles $(L,h_{t,\epsilon})$ converges smoothly to $(L_T,h_T)$ on $\mathcal{R}$ as $t\rightarrow T$ and $\epsilon\rightarrow 0$. \\
\begin{cor}
Let $R>0$, $t_i\rightarrow T$, $\epsilon\rightarrow 0$ and $\tau_i$ be a sequence of holomorphic sections of $L^k$, $k\geq 1$, satisfying
$$
  \int_{M}|\tau_i|^2_{h^k_{h_{t_i,\epsilon_i}}}\omega^n_{t_i,\epsilon_i}\leq 1.
$$
Then, passing to a subsequence if necessary, $\tau_i$ converges to a locally bounded holomorphic section $\tau_{\infty}$ of $L_T^k$ over $\mathcal{R}$ which
satisfies
$$
 \sup_{B_{d_T}(x,r)\cap \mathcal{R}}|\tau_{\infty}|^2_{h_T^{k}}\leq C(R)\cdot r^{-2n}\cdot k^n\cdot \int_{B_{d_T}(x,2r)\cap \mathcal{R}}|\tau_{\infty}|^2_{h_T^{k}}\omega_T^{n}
$$
whenever $B_{d_T}(x,2r)\subset B_{d_T}(x_T,R)$.
\end{cor}
\subsection{Gradient estimate to holomorphic sections}
In this subsection we introduce a family of Hermitian metrics on $L$ which are
$$
h_{FS,\epsilon}=h^{l_0}_{L^{'}}\otimes \left(\frac{\Omega}{(\epsilon^2+|s_D|^2_{h_D})^{1-\beta}}\right)^{-l_0T}\otimes h_D^{l_0T(1-\beta)}\cdot
 e^{-l_0T(1-\beta)\log(\epsilon^2+|s_D|^2_{h_D})}
$$
The metric $h_{FS,\epsilon}$ has curvature $$\Theta_{h_{FS,\epsilon}}=l_0\eta_T. $$ where $\eta_T$ is the induced Fubini-Study metric which satisfies $\eta_T\leq C\omega_{t,\epsilon}$ for some $C$ independent of $t$ and $\epsilon$; see Section 3. An easy calculation shows that for any $t\in [T-\bar{t},T)$,
$$
 h_{t,\epsilon}=e^{-l_0\frac{T}{t}u_{t,\epsilon}}h_{FS,\epsilon},
$$
so $h_{t,\epsilon}$ is uniformly equivalent to $h_{FS,\epsilon}$.

In the following computation we denote $\nabla\tau=\nabla^{h^k_{FS,\epsilon}}\tau$, $\nabla\bar{\nabla}\tau=\nabla^{h^k_{FS,\epsilon}}\bar{\nabla}^{h^k_{FS,\epsilon}}\tau$, and $|\nabla\tau|=|\nabla^{h^k_{FS,\epsilon}}\tau|_{h^k_{FS,\epsilon}\otimes \omega_{t,\epsilon}}$, etc., for any $\tau \in H^0(M,L^k)$, $k\geq 1$.
\begin{lem}
For any $t\in[T-\bar{t},T)$, $\epsilon>0$ and $\tau \in H^0(M,L^k)$, $k\geq 1$, one has
$$
 \triangle|\tau|^2\geq |\nabla\tau|^2-Ck|\tau|^2
$$
and
$$
 \triangle|\nabla\tau|^2\geq |\nabla\nabla\tau|^2+|\nabla\bar{\nabla}\tau|^2-kl_0\nabla_j(\eta_T)_{i\bar{j}}\langle\tau,\nabla_{\bar{i}}\bar{\tau}\rangle
 -kl_0\nabla_{\bar{j}}(tr_{\omega_{t,\epsilon}}\eta_T)\langle\nabla_j\tau,\bar{\tau}\rangle-Ck|\nabla\tau|^2.
$$
\end{lem}
\begin{proof}
They are direct consequences of the Bochner formulas (Lemma (4.1)) and $Ric(\omega_{t,\epsilon})\geq -\frac{1}{t}\omega_{t,\epsilon}$.
\end{proof}
\begin{prop}
For any $R>0$, there exists $C(R)$ independent of $t$, $\epsilon$ and $k\geq 1$ such that for any $t\in[T-\bar{t},T)$ and $B_{\omega_{t,\epsilon}}(x,2r)\subset B_{\omega_{t,\epsilon}}(x_0,R)$, if $\tau \in H^0(B_{\omega_{t,\epsilon}}(x,2r),L^k)$, then
$$
\sup_{B_{\omega_{t,\epsilon}}(x,r)}|\tau|^2_{h^k_{FS,\epsilon}}\leq C(R)\cdot r^{-2n}\cdot k^n\cdot\int_{B_{\omega_{t,\epsilon}}(x,2r)}|\tau|^2_{h^k_{FS,\epsilon}}\omega_{t,\epsilon}^{n}
$$
and
$$
\sup_{B_{\omega_{t,\epsilon}}(x,r)}|\nabla^{h_{FS,\epsilon}^{k}}\tau|^2_{h^k_{FS,\epsilon}\otimes \omega_{t,\epsilon}}\leq C(R)\cdot r^{-2n-2}\cdot k^{n+1}\cdot
\int_{B_{\omega_{t,\epsilon}}(x,2r)}|\tau|^2_{h^k_{FS,\epsilon}}\omega_{t,\epsilon}^{n}.
$$
\end{prop}
The proof of this proposition need to use Lemma (4.17) and Nash-Moser iteration. Because its proof is exactly same as Proposition (3.17) in \cite{NTZ}, we omit it.        \\

In subsection (4.3) we construct a Hermitian line bundle $(L_T,h_T)$ on $\mathcal{R}$. Notice that $h_T=e^{-l_0u_T}h_{FS,\epsilon}=e^{-l_0u_T}\cdot h_{L^{'}}^{l_0}\otimes \Omega^{-l_0T}\otimes h_D^{l_0T(1-\beta)}$. The following lemma is very useful(c.f. Lemma (3.19) \cite{NTZ}).
\begin{lem}
There is a family of cut-off functions $\gamma_\kappa\in C_0^{\infty}(\mathcal{R})$, $\kappa>0$, with $0\leq\gamma_\kappa\leq1$ such that $\gamma_\kappa^{-1}(1)$ forms an exhaustion of $\mathcal{R}$ and, moreover,
$$
 \int_{M_T}|\overline{\partial}\gamma_\kappa|^2\omega_T^{n}\rightarrow 0, \ as \ \kappa\rightarrow 0.
$$
\end{lem}

By a standard iteration we have(c.f.\cite{So})
\begin{prop}
Let $R>0$, $t_i\rightarrow T$, $\epsilon\rightarrow 0$ and $\tau_i$ be a sequence of holomorphic sections of $L^k$, $k\geq 1$, satisfying
$$
  \int_{M}|\tau_i|^2_{h^k_{h_{t_i,\epsilon_i}}}\omega^n_{t_i,\epsilon_i}\leq 1.
$$
Then, passing to a subsequence if necessary, $\tau_i$ converges to a locally bounded holomorphic section $\tau_{\infty}$ of $L_T^k$ over $\mathcal{R}$ which
satisfies
$$
\sup_{B_{d_T}(x,r)\cap \mathcal{R}}|\nabla^{h_{T}^{k}}\tau_{\infty}|^2_{h^k_{T}\otimes \omega_{T}}\leq C(R)\cdot r^{-2n-2}\cdot k^{n+1}\cdot
\int_{{B_{d_T}}(x,2r)\cap \mathcal{R}}|\tau_{\infty}|^2_{h^k_{T}}\omega_{T}^{n}.
$$
whenever $B_{d_T}(x,2r)\subset B_{d_T}(x_T,R)$.
\end{prop}

\subsection{Algebraic structure of $M_T$}
Recall that if $\tau \in H^0(M,L^k)$, then by the construction of $(L_T,h_T)$ on $\mathcal{R}$, one knows that $\tau|_{\mathcal{R}}$ denoted by $\tau_{\infty}$ is a holomorphic section of $(L_T,h_T)$. For a fixed $\tau \in H^0(M,L^k)$, one has
$$
\int_{M}|\tau|^2_{h_{t,\epsilon}}\omega_{t,\epsilon}^n\leq C\int_{M}|\tau|^2_{h_{FS,\epsilon}}\frac{\Omega}{|s_D|_{h_D}^{2(1-\beta)}}\leq C_{\tau}.
$$
Therefore, by $h_{t,\epsilon}\xrightarrow{C^{\infty}} h_T$ and Lemma (4.14), we have $$\sup_{B_{d_T}(x,r)\cap \mathcal{R}}|\tau_{\infty}|^2_{h_T^{k}}\leq C(R,r,k)$$

Notice that
\begin{align*}
  |\nabla^{h_T^{k}}\tau_{\infty}|_{h_T^{k}\otimes \omega_T}&\leq |\nabla^{h_{FS,\epsilon}^{k}}\tau_{\infty}|_{h_T^{k}\otimes \omega_T}+kl_0|\tau_{\infty}|_{h_T^{k}}\cdot |\nabla u_T|_{\omega_T}  \\
  &\leq C|\nabla^{h_{FS,\epsilon}^{k}}\tau_{\infty}|_{h_{FS,\epsilon}^{k}\otimes \eta_T}+kl_0|\tau_{\infty}|_{h_T^{k}}\cdot |\nabla u_T|_{\omega_T}.
\end{align*}
where the last inequality base on the estimate $\omega_T\geq C^{-1}\eta_T$ and the fact that $h_T$ is equivalent to $h_{FS,\epsilon}$.

From Proposition (3.9), $|\tau_{\infty}|_{h_T^{k}}\cdot |\nabla u_T|_{\omega_T}$ is bounded on $B_{d_T}(x,r)\cap \mathcal{R}$. By the Song's argument(Lemma (3.10) in \cite{So}) one gets
$|\nabla^{h_{FS,\epsilon}^{k}}\tau_{\infty}|_{h_{FS,\epsilon}^{k}\otimes \eta_T}$ is also bounded on $B_{d_T}(x,r)\cap \mathcal{R}$.
Thus
$$
 \sup_{B_{d_T}(x,r)\cap \mathcal{R}} |\nabla^{h_T^{k}}\tau_{\infty}|_{h_T^{k}\otimes \omega_T}\leq C(R,r,k),
$$
i.e. $\tau_{\infty}$ with metric $h_T$ is locally Lipschitz, moreover it can be continuously extended to $M_T$.

So, the map
$$
  \Phi_{T}: \ (\mathcal{R},d_T)\rightarrow (\Phi(M),\omega_{FS})
$$
defined by $\Phi$ can be continuously extended to
$$
 \Phi_{T}: \ (M_T,d_T)\rightarrow (\Phi(M),\omega_{FS})
$$
that is a Lipschitz map, since $\Phi_{T}^{*}\omega_{FS}=kl_0\eta_T\leq Ckl_0\omega_T$.
\begin{prop}
$\Phi_T$ is injective and is a local homeomorphism.
\end{prop}
The proof of this Proposition is exactly same as Proposition (3.21) and Proposition (3.22) in \cite{NTZ} so we omit it.
\section{Diameter bound of the conical K\"{a}hler metric}
Let $\omega_{T}$ be the solution to the following equation in the current sense
$$
 (\omega_T+\sqrt{-1}\partial\bar{\partial}u_{T})^n=e^{\frac{u_T}{T}}\frac{\Omega}{|s_D|^{2(1-\beta)}_{h_D}}.
$$
In \cite{So},Song developed a method to prove the diameter bound of a singular K\"{a}hler-Einstein metric. In this subsection, we follow his idea to show the diameter bound of $(M\backslash (D\cup \bar{D}), \omega_T)$ where $\bar{D}$ is any divisor such that $[\omega_0]-Tc_1(M)+T(1-\beta)c_{1}(L_D)-\mu c_1(L_{\bar{D}})>0$ for some $\mu>0$. We will consider the following three cases.   \\

\textbf{Case} 1. If $p\in D\backslash \bar{D}$, then by Theorem (1.4) there exists a neighborhood $U$ of $p$ such that $$\omega_T\leq C_{U}\frac{\omega_0}{|s_D|^{2(1-\beta)}_{h_D}}, \  on \  U$$.  \\

\textbf{Case} 2.

Let $p\in D\cap \bar{D}$ be any point, $\pi: \ \widetilde{M}\rightarrow M$ be the blow-up at $p$ with exceptional divisor $\pi^{-1}(p)=E$. Then
$$
 K_{\widetilde{M}}=\pi^{*}K_{M}+(n-1)E.
$$
Let $h_E$ be the Hermitian metric on $L_E$ associated with the divisor $E$, and $\sigma_{E}$ be a defining section. We denote by $D_1=\overline{\pi^{-1}(D)-E}$, $h_{D_1}=\pi^{*}h_D$ and $D_2=\overline{\pi^{-1}(\bar{D})-E}$, $h_{D_2}=\pi^{*}h_{\bar{D}}$. Let $\chi$ be a fixed K\"{a}hler metric on $\widetilde{M}$. Let $\sigma_{D_1}$ be a defining section on $L_{D_1}$ and $\sigma_{D_2}$ be a defining section on $L_{D_2}$.By the calculation one has
$$
\pi^{*}\eta_T+\mu\sqrt{-1}\partial\bar{\partial}\log|\sigma_{D_2}|^2_{h_{D_2}}+\delta_0\sqrt{-1}\partial\bar{\partial}\log|\sigma_E|^2_{h_E}\geq \delta_1 \chi
$$
for some small $\delta_0,\delta_1>0$ on $\widetilde{M}\backslash (D_2\cup E)$. Observe that $\tilde{\Omega}=|\sigma_E|^{-2(n-1)}_{h_E}\pi^{*}\Omega$ defines a smooth volume form on $\widetilde{M}$.
Consider the following family of Monge-Amp\`{e}re equations on $\widetilde{M}$
\begin{equation}
(\pi^{*}\eta_T+\epsilon\chi+\sqrt{-1}\partial\bar{\partial}\widetilde{\varphi_{\epsilon,\delta}})^n=
e^{\frac{1}{T}\widetilde{\varphi_{\epsilon,\delta}}}(\epsilon^2+|\sigma_E|^2_{h_E})^{n-1}\frac{\tilde{\Omega}}{(\delta^2+|\sigma_{D_{1}}|^2_{h_{D_1}})^{1-\beta}}.
\end{equation}
By Yau's solution to Calabi conjecture \cite{Y2}, the equation has a unique smooth solution $\widetilde{\varphi_{\epsilon,\delta}}$; moreover
$$
\widetilde{\omega_{\epsilon,\delta}}=\pi^{*}\eta_T+\epsilon\chi+\sqrt{-1}\partial\bar{\partial}\widetilde{\varphi_{\epsilon,\delta}}
$$
is a smooth K\"{a}hler metric on $\widetilde{M}$.
\begin{lem}
For any $\mu>0$ and $\delta_0>0$, there exist $C(\mu,\delta_0)$ and $C$ independent of $\epsilon$ and $\delta$ such that
$$
 \mu\sqrt{-1}\partial\bar{\partial}\log|\sigma_{D_2}|^2_{h_{D_2}}+\delta_0\sqrt{-1}\partial\bar{\partial}\log|\sigma_E|^2_{h_E}-C(\mu,\delta_0)\leq \widetilde{\varphi_{\epsilon,\delta}}\leq C.
$$
\end{lem}
\begin{proof}
We follow Song's argument \cite{So2}. For upper bound, let
$$
V_{\epsilon,\delta}=\int_{\widetilde{M}}(\epsilon^2+|\sigma_E|^2_{h_E})^{n-1}\frac{\tilde{\Omega}}{(\delta^2+|\sigma_{D_1}|^2_{h_{D_1}})^{1-\beta}}
$$
be the volume. We see that
$$
V_{1,0}\geq V_{\epsilon,\delta}\geq V_{0,1}=\int_{\widetilde{M}}\frac{\tilde{\Omega}}{|\sigma_{D_1}|^{2(1-\beta)}_{h_{D_1}}}
$$
hence $V_{\epsilon,\delta}$ is uniformly bounded. We denote $\widetilde{\Omega_{\epsilon,\delta}}=(\epsilon^2+|\sigma_E|^2_{h_E})^{n-1}\frac{\tilde{\Omega}}{(\delta^2+|\sigma_{D_1}|^2_{h_{D_1}})^{1-\beta}}$, then we have
the following calculation
\begin{align*}
\frac{1}{V_{\epsilon,\delta}}\int_{\widetilde{M}}\frac{1}{T}\widetilde{\varphi_{\epsilon,\delta}}\widetilde{\Omega_{\epsilon,\delta}}&=\frac{1}{V_{\epsilon,\delta}}
\int_{\widetilde{M}}\log\left(\frac{\widetilde{\omega_{\epsilon,\delta}}^n}{\widetilde{\Omega_{\epsilon,\delta}}}\right)\cdot\widetilde{\Omega_{\epsilon,\delta}}\\
 &\leq \log\int_{\widetilde{M}}\widetilde{\omega_{\epsilon,\delta}}^n-\log V_{\epsilon,\delta}\\
 &=\log(\int_{\widetilde{M}}(\pi^*\eta_T+\epsilon\chi)^n) -C \leq C
\end{align*}
where for the first inequality we use Jensen¨s inequality. Since $\widetilde{\varphi_{\epsilon,\delta}}\in PSH(\widetilde{M},\pi^*\eta_T+\epsilon\chi)$,the mean value inequality implies that $$\sup_{\widetilde{M}}\widetilde{\varphi_{\epsilon,\delta}}\leq C.$$

For the lower bound, we set $\widetilde{\varphi_{\epsilon,\delta}}^{'}=\widetilde{\varphi_{\epsilon,\delta}}-\mu\sqrt{-1}\partial\bar{\partial}\log|\sigma_{D_2}|^2_{h_{D_2}}
-\delta_0\sqrt{-1}\partial\bar{\partial}\log|\sigma_E|^2_{h_E}$, then by (5.1) one knows
\begin{multline}
((\pi^{*}\eta_T+\mu\sqrt{-1}\partial\bar{\partial}\log|\sigma_{D_2}|^2_{h_{D_2}}+\delta_0\sqrt{-1}\partial\bar{\partial}\log|\sigma_E|^2_{h_E})+\epsilon\chi
+\sqrt{-1}\partial\bar{\partial}\widetilde{\varphi_{\epsilon,\delta}}^{'})^n=  \\
e^{\frac{1}{T}\widetilde{\varphi_{\epsilon,\delta}}^{'}}\cdot \frac{1}{T}|\sigma_{D_{2}}|_{h_{D_2}}^{2\mu}\cdot \frac{1}{T}|\sigma_E|^{2\delta_0}_{h_E}
\cdot (|\sigma_E|^2_{h_E}+\epsilon^2)^{n-1}\cdot \frac{\tilde{\Omega}}{(|\sigma_{D_1}|^2_{h_{D_1}}+\delta^2)^{1-\beta}}.
\end{multline}
We consider the following Monge-Amp\'{e}re equations
$$
(\pi^{*}\eta_T-\mu Ric(h_{D_2})-\delta_0 Ric(h_E)+\epsilon\chi+\sqrt{-1}\partial\bar{\partial}\psi_{\epsilon,\delta})^n=e^{\frac{1}{T}\psi_{\epsilon,\delta}}\cdot
(|\sigma_E|^2_{h_E}+\epsilon^2)^{n-1}\cdot\frac{\tilde{\Omega}}{(|\sigma_{D_1}|^2_{h_{D_1}}+\delta^2)^{1-\beta}}.
$$
By Yau's theorem \cite{Y2}, the above equation admits a unique smooth solution. By \cite{EGZ}, we have
$$
  |\psi_{\epsilon,\delta}|_{C^0}\leq C(\mu,\delta_0).
$$
Set $H_{\epsilon,\delta}=\widetilde{\varphi_{\epsilon,\delta}}'-\psi_{\epsilon,\delta}$ and $\nu_\epsilon=\pi^*\eta_T-\mu Ric(h_{D_2})-\delta_0 Ric(h_E)+\epsilon\chi$, then on $\widetilde{M}\backslash (E\cup D_1\cup D_2)$ one knows
$$
\log\frac{(\nu_\epsilon+\sqrt{-1}\partial\bar{\partial}\psi_{\epsilon,\delta}+\sqrt{-1}\partial\bar{\partial}H_{\epsilon,\delta})^n}
{(\nu_\epsilon+\sqrt{-1}\partial\bar{\partial}\psi_{\epsilon,\delta})^n}=\frac{1}{T}H_{\epsilon,\delta}-2\log T+\log|\sigma_{D_2}|^{2\mu}_{h_{D_2}}+\log
|\sigma_E|^{2\delta_0}_{h_E}.
$$
The minimum of $H_{\epsilon,\delta}$ cannot be at $D_2\cup E$. Assume $H_{\epsilon,\delta}$ attains minimum at $x_0$, then by maximum principle, one gets
$$
(\frac{1}{T}H_{\epsilon,\delta}-2\log T+\log|\sigma_{D_2}|^{2\mu}_{h_{D_2}}+\log |\sigma_E|^{2\delta_0}_{h_E})(x_0)\geq 0.
$$
Hence we know
$$
\inf_{\widetilde{M}}H_{\epsilon,\delta}\geq -C.
$$
By the $C^0$ estimate of $\psi_{\epsilon,\delta}$, we obtain the lower bound of $\widetilde{\varphi_{\epsilon,\delta}}$.
\end{proof}
\begin{lem}
There exist $C$ and $\lambda_1$ independent of $\epsilon$ and $\delta$ such that
$$
 Ric(\widetilde{\omega_{\epsilon,\delta}})\leq -\frac{1}{T}\widetilde{\omega_{\epsilon,\delta}}+C\frac{\chi}{|\sigma_{D_1}|^{2\lambda_1}_{h_{D_1}}}.
$$
\end{lem}
\begin{proof}
First, we observe following facts:
\begin{enumerate}
\item  Since $\tilde{\Omega}$ is a smooth volume form, $Ric(\tilde{\Omega})\leq C\chi$.\\
\item  $\sqrt{-1}\partial\bar{\partial}\log(|\sigma_E|^2_{h_E}+\epsilon^2)^{n-1}\geq -C\chi$.\\
\item  $\pi^{*}\eta_T\leq C\chi$.\\
\item  If $\lambda_1$ is sufficiently large, one has $\sqrt{-1}\partial\bar{\partial}\log(|\sigma_{D_1}|_{h_{D_1}}^2+\delta^2)\leq \frac{C\chi}{|\sigma_{D_1}|_{h_{D_1}}^{2\lambda_1}}$.
\end{enumerate}
Thus by a simple calculation one gets
$$
Ric(\widetilde{\omega_{\epsilon,\delta}})\leq -\frac{1}{T}\widetilde{\omega_{\epsilon,\delta}}+C\frac{\chi}{|\sigma_{D_1}|^{2\lambda_1}_{h_{D_1}}}.
$$
\end{proof}
\begin{lem}
There exist $C$ and $\lambda$ independent of $\epsilon$ and $\delta$ such that
$$
 \widetilde{\omega_{\epsilon,\delta}}\leq \frac{C}{|\sigma_E|^{2\lambda}_{h_E}|\sigma_{D_1}|^{2\lambda}_{h_{D_1}}|\sigma_{D_2}|^{2\lambda}_{h_{D_2}}}\chi.
$$
\end{lem}
\begin{proof}
By a standard calculation one has
$$
 \triangle_{\widetilde{\omega_{\epsilon,\delta}}}\log tr_{\chi}\widetilde{\omega_{\epsilon,\delta}}\geq -Ctr_{\widetilde{\omega_{\epsilon,\delta}}}\chi
 -\frac{C}{|\sigma_{D_1}|^{2\lambda_1}_{h_{D_1}}tr_{\chi}\widetilde{\omega_{\epsilon,\delta}}}.
$$
There is a easy fact that is
$$
 \triangle_{\widetilde{\omega_{\epsilon,\delta}}}\widetilde{\varphi_{\epsilon,\delta}}=n-tr_{\widetilde{\omega_{\epsilon,\delta}}}\pi^{*}\eta_T-\epsilon
 tr_{\widetilde{\omega_{\epsilon,\delta}}}\chi.
$$
Let $H=\log(|\sigma_E|^{2A}_{h_E}|\sigma_{D_1}|^{2A}_{h_{D_1}}|\sigma_{D_2}|^{2A}_{h_{D_2}}tr_{\chi}\widetilde{\omega_{\epsilon,\delta}})-A^2\widetilde{\varphi_{\epsilon,\delta}}$.
Then, on $\tilde{M}\backslash (E\cup D_1\cup D_2)$, we get
$$
\triangle_{\widetilde{\omega_{\epsilon,\delta}}}H\geq -Ctr_{\widetilde{\omega_{\epsilon,\delta}}}\chi-\frac{C}{|\sigma_{D_1}|^{2\lambda_1}_{h_{D_1}}tr_{\chi}\widetilde{\omega_{\epsilon,\delta}}}-A^2n+
Atr_{\widetilde{\omega_{\epsilon,\delta}}}(A\pi^{*}\eta_T-Ric(h_E)-Ric(h_{D_1})-Ric(h_{D_2})).
$$
Notice that when $A$ is sufficiently large we observe that
$$
Atr_{\widetilde{\omega_{\epsilon,\delta}}}(A\pi^{*}\eta_T-Ric(h_E)-Ric(h_{D_1})-Ric(h_{D_2}))\geq (C+1)tr_{\widetilde{\omega_{\epsilon,\delta}}}\chi.
$$
Therefore
$$
\triangle_{\widetilde{\omega_{\epsilon,\delta}}}H\geq tr_{\widetilde{\omega_{\epsilon,\delta}}}\chi-\frac{C}{|\sigma_{D_1}|^{2\lambda_1}_{h_{D_1}}tr_{\chi}\widetilde{\omega_{\epsilon,\delta}}}-A^2n.
$$
Assume that $H$ attains maximum at $x_0$ ($x_0\in \widetilde{M}\backslash (E\cup D_1\cup D_2)$), one deduces
$$
 (|\sigma_{D_1}|^{2\lambda_1}_{h_{D_1}}tr_{\chi}\widetilde{\omega_{\epsilon,\delta}})(tr_{\widetilde{\omega_{\epsilon,\delta}}}\chi-A^2n)(x_0)\leq C
$$
Using an inequality ${\widetilde{\omega_{\epsilon,\delta}}}^n\leq C\frac{\chi^n}{|\sigma_{D_1}|^{2(1-\beta)}_{h_{D_1}}}$, one obtains
$$
\frac{1}{C}|\sigma_{D_1}|^{\frac{2(1-\beta)}{n-1}}_{h_{D_1}}(tr_{\chi}\widetilde{\omega_{\epsilon,\delta}})^{\frac{1}{n-1}} \leq tr_{\widetilde{\omega_{\epsilon,\delta}}}\chi
$$
Thus
\begin{equation}
(|\sigma_{D_1}|^{2\lambda_1}_{h_{D_1}}tr_{\chi}\widetilde{\omega_{\epsilon,\delta}})
(\frac{1}{C}|\sigma_{D_1}|^{\frac{2(1-\beta)}{n-1}}_{h_{D_1}}(tr_{\chi}\widetilde{\omega_{\epsilon,\delta}})^{\frac{1}{n-1}}-A^2n)(x_0)\leq C.
\end{equation}
If $$(tr_{\chi}\widetilde{\omega_{\epsilon,\delta}})^{\frac{1}{n-1}}(x_0)\leq \frac{2CA^2n}{|\sigma_{D_1}|^{\frac{2(1-\beta)}{n-1}}_{h_{D_1}}}(x_0),$$ then there exists $\lambda_2$ such that $$tr_{\chi}\widetilde{\omega_{\epsilon,\delta}}(x_0)\leq \frac{C}{|\sigma_{D_1}|_{h_{D_1}}^{2\lambda_2}}(x_0).$$
Otherwise, $$(tr_{\chi}\widetilde{\omega_{\epsilon,\delta}})^{\frac{1}{n-1}}(x_0)\geq \frac{2CA^2n}{|\sigma_{D_1}|^{\frac{2(1-\beta)}{n-1}}_{h_{D_1}}}(x_0),$$
from (5.6) one knows that $$tr_{\chi}\widetilde{\omega_{\epsilon,\delta}}(x_0)\leq \frac{C}{|\sigma_{D_1}|_{h_{D_1}}^{2\lambda_1}}(x_0).$$
In general one can find $\lambda^{'}$ such that $$tr_{\chi}\widetilde{\omega_{\epsilon,\delta}}(x_0)\leq \frac{C}{|\sigma_{D_1}|_{h_{D_1}}^{2\lambda^{'}}}(x_0).$$
Choose $A>>\lambda^{'}$, one knows $H\leq C$. Therefore the Lemma is proved.
\end{proof}

Let $B$ be a disk centered at $p$ and $\tilde{B} = \pi^{-1}(B)$. Denote $f_1, \cdots , f_N$ as the defining functions of divisors $D_1$ and $D_2$.
\begin{cor}
There exist $C$ and $\lambda$ independent of $\epsilon$ and $\delta$ such that
$$
\widetilde{\omega_{\epsilon,\delta}}|_{\partial \tilde{B}}\leq C(\prod_{i=1}^{N}|f_i|^{-2\lambda}\chi)|_{\partial \tilde{B}}.
$$
\end{cor}

Let $\hat{\chi}$ be the smooth closed nonnegative closed $(1,1)$-form as the pullback of the Euclidean metric $\sqrt{-1}\sum_{j=1}^{n}dz_j\wedge d\bar{z_j}$
on $B$. $\hat{\chi}$ is a k\"{a}hler metric on $\tilde{B}\backslash E$.
\begin{lem}
There exists $C>0$, a sufficiently small $\epsilon_0>0$ and a smooth Hermitian metric $h_E$ on $L_E$ such that in $\tilde{B}$
$$
 C^{-1}\hat{\chi}\leq \chi\leq C\frac{\hat{\chi}}{|\sigma_E|^2_{h_E}} ,
$$
$$
 \pi^{*}\eta_T-\epsilon_0Ric(h_E)>0.
$$
\end{lem}

The following proposition is the main result of this section.
\begin{prop}
There exist $0<\alpha<1$, $\lambda$ and $C>0$ independent of $\epsilon$ and $\delta$ such that
$$
\widetilde{\omega_{\epsilon,\delta}}\leq \frac{C}{|\sigma_E|^{2(1-\alpha)}_{h_E}\prod_{i=1}^{N}|f_i|^{2\lambda}}\chi, \ in \ \tilde{B}.
$$
\end{prop}
\begin{proof}
Let $H_{\epsilon,\delta}=\log (|\sigma_E|^{2(1+r)}_{h_E}\cdot\prod_{i=1}^{N}|f_i|^{2\lambda}\cdot tr_{\hat{\chi}}\widetilde{\omega_{\epsilon,\delta}})-A\widetilde{\varphi_{\epsilon,\delta}}$ for
some sufficiently large $A$ and sufficiently small $r$. There are some facts in $\tilde{B}\backslash(E\cup D_1\cup D_2)$:
\begin{enumerate}
\item  $\triangle_{\widetilde{\omega_{\epsilon,\delta}}}\log|\sigma_E|^2_{h_E}=-tr_{\widetilde{\omega_{\epsilon,\delta}}}(Ric(h_E))$,\\
\item  $\triangle_{\widetilde{\omega_{\epsilon,\delta}}}\log \prod_{i=1}^{N}|f_i|^{2\lambda}=0$, \\
\item $\triangle_{\widetilde{\omega_{\epsilon,\delta}}}\widetilde{\varphi_{\epsilon,\delta}}=n-tr_{\widetilde{\omega_{\epsilon,\delta}}}\pi^{*}\eta_T-\epsilon
 tr_{\widetilde{\omega_{\epsilon,\delta}}}\chi$, \\
\item  $\triangle_{\widetilde{\omega_{\epsilon,\delta}}} \log tr_{\hat{\chi}}\widetilde{\omega_{\epsilon,\delta}}\geq -C tr_{\widetilde{\omega_{\epsilon,\delta}}}\chi-C(|\sigma_E|^2_{h_E}|\sigma_{D_1}|^{2\lambda_1}_{h_{D_1}}tr_{\hat{\chi}}\widetilde{\omega_{\epsilon,\delta}})^{-1}$.
\end{enumerate}
Thus in $\tilde{B}\backslash(E\cup D_1\cup D_2)$ one has
\begin{align*}
\triangle_{\widetilde{\omega_{\epsilon,\delta}}}H_{\epsilon,\delta}&\geq -C tr_{\widetilde{\omega_{\epsilon,\delta}}}\chi-\frac{C}{|\sigma_E|^2_{h_E}|\sigma_{D_1}|^{2\lambda_1}_{h_{D_1}}tr_{\hat{\chi}}\widetilde{\omega_{\epsilon,\delta}}}
-An-(r+1)tr_{\widetilde{\omega_{\epsilon,\delta}}}(Ric(h_E))+A tr_{\widetilde{\omega_{\epsilon,\delta}}}\pi^{*}\eta_T \\
&\geq tr_{\widetilde{\omega_{\epsilon,\delta}}}\chi-\frac{C}{|\sigma_E|^2_{h_E}|\sigma_{D_1}|^{2\lambda_1}_{h_{D_1}}tr_{\hat{\chi}}\widetilde{\omega_{\epsilon,\delta}}}
-An
\end{align*}
where the last inequality base on Lemma (5.8) and the sufficiently large number $A$.

By a similar calculation one gets
$$
\triangle_{\widetilde{\omega_{\epsilon,\delta}}}\log tr_{\chi}\widetilde{\omega_{\epsilon,\delta}}\geq -C_1 tr_{\widetilde{\omega_{\epsilon,\delta}}}\chi
 -\frac{C_1}{|\sigma_{D_1}|^{2\lambda_1}_{h_{D_1}}tr_{\chi}\widetilde{\omega_{\epsilon,\delta}}}.
$$
Let $G_{\epsilon,\delta}=H_{\epsilon,\delta}+\frac{1}{2C_1}\log \prod_{i=1}^{N}|f_i|^{2\lambda+2}tr_{\chi}\widetilde{\omega_{\epsilon,\delta}}$. By the same argument one knows
$$
\triangle_{\widetilde{\omega_{\epsilon,\delta}}}G_{\epsilon,\delta}\geq \frac{1}{2}tr_{\widetilde{\omega_{\epsilon,\delta}}}\chi-An-\frac{1}{2|\sigma_{D_1}|^{2\lambda_1}_{h_{D_1}}tr_{\chi}\widetilde{\omega_{\epsilon,\delta}}}
-\frac{C}{|\sigma_E|^2_{h_E}|\sigma_{D_1}|^{2\lambda_1}_{h_{D_1}}tr_{\hat{\chi}}\widetilde{\omega_{\epsilon,\delta}}}.
$$
By Lemma (5.8) and the above inequality, we have
$$
 \triangle_{\widetilde{\omega_{\epsilon,\delta}}}G_{\epsilon,\delta}\geq \frac{1}{2}tr_{\widetilde{\omega_{\epsilon,\delta}}}\chi-An
-\frac{C}{|\sigma_E|^2_{h_E}|\sigma_{D_1}|^{2\lambda_1}_{h_{D_1}}tr_{\hat{\chi}}\widetilde{\omega_{\epsilon,\delta}}}.
$$
For fixed sufficiently large $\lambda>0$, there exists $C>0$ such that
$$
\sup_{\partial\tilde{B}}G_{\epsilon,\delta}\leq C
$$
from the estimate in Corollary (5.7).

So we assume that
$$
 \sup_{\tilde{B}}G_{\epsilon,\delta}=G_{\epsilon,\delta}(p_{max})
$$
for some $p_{max}\in \tilde{B}\backslash (E\cup D_1\cup D_2)$. Then at $p_{max}$
$$
(tr_{\widetilde{\omega_{\epsilon,\delta}}}\chi-2An)\cdot|\sigma_{D_1}|^{2\lambda_1}_{h_{D_1}}|\sigma_E|^2_{h_E}\cdot tr_{\hat{\chi}}
\widetilde{\omega_{\epsilon,\delta}}(p_{max})\leq C.
$$
Notice that
$$
\frac{1}{C}|\sigma_{D_1}|^{\frac{2(1-\beta)}{n-1}}_{h_{D_1}}(tr_{\chi}\widetilde{\omega_{\epsilon,\delta}})^{\frac{1}{n-1}} \leq tr_{\widetilde{\omega_{\epsilon,\delta}}}\chi.
$$
Then we have
\begin{equation}
(\frac{1}{C}|\sigma_{D_1}|^{\frac{2(1-\beta)}{n-1}}_{h_{D_1}}(tr_{\chi}\widetilde{\omega_{\epsilon,\delta}})^{\frac{1}{n-1}}-2An)
\cdot|\sigma_{D_1}|^{2\lambda_1}_{h_{D_1}}|\sigma_E|^2_{h_E}\cdot tr_{\hat{\chi}}
\widetilde{\omega_{\epsilon,\delta}}(p_{max})\leq C.
\end{equation}
If $$(tr_{\chi}\widetilde{\omega_{\epsilon,\delta}})^{\frac{1}{n-1}}(p_{max})\leq \frac{3CAn}{|\sigma_{D_1}|^{\frac{2(1-\beta)}{n-1}}_{h_{D_1}}}(p_{max}),$$ then we observe that $$tr_{\hat{\chi}}\widetilde{\omega_{\epsilon,\delta}}(p_{max})\leq \frac{C}{|\sigma_E|^2_{h_E}|\sigma_{D_1}|^{2(1-\beta)}_{h_{D_1}}}(p_{max}). $$ Hence $G_{\epsilon,\delta}$ is bounded above by a uniform constant.

Otherwise $$(tr_{\chi}\widetilde{\omega_{\epsilon,\delta}})^{\frac{1}{n-1}}(p_{max})\geq \frac{3CAn}{|\sigma_{D_1}|^{\frac{2(1-\beta)}{n-1}}_{h_{D_1}}}(p_{max}),$$ i.e. $$An\leq \frac{1}{3C}|\sigma_{D_1}|^{\frac{2(1-\beta)}{n-1}}_{h_{D_1}}
\cdot (tr_{\chi}\widetilde{\omega_{\epsilon,\delta}})^{\frac{1}{n-1}}(p_{max})$$
Then by (5.10) one gets
$$
\log|\sigma_E|^2_{h_E}+\log tr_{\hat{\chi}}\widetilde{\omega_{\epsilon,\delta}}+\frac{1}{n-1}\log tr_{\chi}\widetilde{\omega_{\epsilon,\delta}}
+\log|\sigma_{D_1}|_{h_{D_1}}^{2\lambda_1+\frac{2(1-\beta)}{n-1}}(p_{max})\leq C.
$$
Moreover, combining with the Lemma (5.2) and choosing sufficiently large $\lambda$ one knows
$$
G_{\epsilon,\delta}(p_{max})\leq C
$$
In sun, in all cases, we have $G_{\epsilon,\delta}\leq C$. Then
$$
\log(|\sigma_E|_{h_E}^{2(1+r)}\prod_{i=1}^{N}|f_i|^{2\lambda+\frac{2\lambda+2}{2C_1}}\cdot (tr_{\hat{\chi}}\widetilde{\omega_{\epsilon,\delta}})\cdot
(tr_{\chi}\widetilde{\omega_{\epsilon,\delta}})^{\frac{1}{2C_1}})\leq C.
$$
Noting that $tr_{\hat{\chi}}\widetilde{\omega_{\epsilon,\delta}}\geq C^{-1}tr_{\chi}\widetilde{\omega_{\epsilon,\delta}}$, we have
$$
(tr_{\chi}\widetilde{\omega_{\epsilon,\delta}})^{1+\frac{1}{2C_1}}\leq \frac{C}{|\sigma_E|_{h_E}^{2(1+r)}\prod_{i=1}^{N}|f_i|^{2\Lambda}}.
$$
If we choose $r=\frac{1}{10C_1}$, then $\frac{1+r}{1+(2C_1)^{-1}}=1-\alpha$ for some $\alpha\in (0,1)$. The Proposition is proved.
\end{proof}
\begin{cor}
Assume as above. There exist $\alpha>0$, $\lambda>0$ and $C>0$ such that
$$
 \pi^{*}\omega_T\leq \frac{C}{|\sigma_E|^{2(1-\alpha)}_{h_E}\prod_{i=1}^{N}|f_i|^{2\lambda}}\chi, \ in \ \tilde{B}.
$$
\end{cor}

\textbf{Case} 3.

Let $p\in \bar{D}\backslash D$ be any point and $\pi$, $\widetilde{M}$, $E$, $h_E$, $\sigma_E$, $D_2$, $h_{D_2}$, $\sigma_{D_2}$, $\chi$ and $\tilde{\Omega}$ be the same as \textbf{Case 2}. Consider the following family of Monge-Amp\`{e}re equations on $\widetilde{M}$
$$
(\pi^{*}\eta_T+\epsilon\chi+\sqrt{-1}\partial\bar{\partial}\widetilde{\varphi_{\epsilon,\delta}})^n=
e^{\frac{1}{T}\widetilde{\varphi_{\epsilon,\delta}}}(\epsilon^2+|\sigma_E|^2_{h_E})^{n-1}\frac{\tilde{\Omega}}{(\delta^2+|s_D|^2_{h_D})^{1-\beta}}.
$$
By Yau's solution to Calabi conjecture \cite{Y2}, the equation has a unique smooth solution $\widetilde{\varphi_{\epsilon,\delta}}$; moreover
$$
\widetilde{\omega_{\epsilon,\delta}}=\pi^{*}\eta_T+\epsilon\chi+\sqrt{-1}\partial\bar{\partial}\widetilde{\varphi_{\epsilon,\delta}}
$$
is a smooth K\"{a}hler metric on $\widetilde{M}$.

Let $B$ be a disk centered at $p$ such that $B\cap D=\emptyset$ and $\tilde{B}=\pi^{-1}(B)$. Denote $f_1, \cdots , f_{N_1}$ as the defining functions of divisor $D_2$. By the same argument of Proposition (5.9) we have
\begin{cor}
There exist $0<\alpha<1$, $\lambda$ and $C>0$ independent of $\epsilon$ and $\delta$ such that
$$
\pi^{*}\omega_T\leq \frac{C}{|\sigma_E|^{2(1-\alpha)}_{h_E}\prod_{i=1}^{N_1}|f_i|^{2\lambda}}\chi, \ in \ \tilde{B}.
$$
\end{cor}
From now on we turn to the Gromov-Hausdorff convergence. By the argument of \cite{NTZ}, Corollary (5.11) and Corollary (5.12), we immediately conclude the following proposition.
\begin{prop}
$\Phi_{T}: \ M_T\rightarrow \Phi(M)$ s a homeomorphism. As a consequence, the diameter of $M_T$ is finite. Furthermore, there exists $C$
such that $$diam(M,\omega_t)\leq C, \ \forall \ t\in[T-\bar{t},T).$$
\end{prop}

\end{document}